\documentclass[11pt]{amsart}

\usepackage{amsmath,amsthm,verbatim,amssymb,amsfonts,amscd, graphicx,cite,url,enumerate,hyperref,bbm}
\usepackage{graphics}
\hypersetup{
colorlinks=true,
linkcolor=blue,
citecolor=black
}
\theoremstyle{plain}
\newtheorem{theorem}{Theorem}
\newtheorem{corollary}{Corollary}
\newtheorem{lemma}[corollary]{Lemma}
\newtheorem{proposition}[corollary]{Proposition}

\theoremstyle{definition}

\theoremstyle{definition}

\newtheorem*{acknow}{Acknowledgements}

\begin{document}
	
	\title[Low pseudomoments of the Riemann zeta function and its powers]{Low pseudomoments of the Riemann \\ zeta function and its powers}
	\author{Maxim Gerspach}
	\thanks{The author was supported by DFG-SNF lead agency program grant 200020L\textunderscore175755 and an associated SNF mobility in project grant}
	\date{\today}
	\address{ETH Z\"urich, Switzerland}
	\email{maxim.gerspach@math.ethz.ch}
	\maketitle
	
	\begin{abstract}
		The $2 q$-th pseudomoment $\Psi_{2q,\alpha}(x)$ of the $\alpha$-th power of the Riemann zeta function is defined to be the $2 q$-th moment of the partial sum up to $x$ of $\zeta^\alpha$ on the critical line. 
		Using probabilistic methods of Harper, we prove upper and lower bounds for these pseudomoments when $q \le \frac{1}{2}$ and $\alpha \ge 1$.
		Combined with results of Bondarenko, Heap and Seip, these bounds determine the size of all pseudomoments with $q > 0$ and $\alpha \ge 1$ up to powers of $\log \log x$, where $x$ is the length of the partial sum, and it turns out that there are three different ranges with different growth behaviours. In particular, the results give the order of magnitude of $\Psi_{2 q, 1}(x)$ for all $q > 0$.
	\end{abstract}
	
	\section{Introduction}\label{Intro}
	
	For $\alpha \in \mathbb{C}$, let $d_\alpha$ denote the generalised divisor functions, i.e.
	\[ \zeta(s)^\alpha = \sum_{n \ge 1} \frac{d_\alpha(n)}{n^s} \]
	for $\Re s > 1$. Further, we write
	\[ \Psi_{2 q, \alpha}(x) := \lim_{T \to \infty} \frac{1}{T} \int_0^T \left| \sum_{n \le x} \frac{d_\alpha(n)}{n^{1/2 + i t}} \right|^{2 q} \,dt \]
	and $\Psi_{2 q}(x) := \Psi_{2 q, 1}(x)$. The functions $\Psi_{2 q}(x)$ are called the \textit{pseudomoments} of the Riemann zeta function, and were first defined by Conrey and Gamburd \cite{ConreyGamburd1}. They proved that for $q \in \mathbb{N}$ we have
	\[ \Psi_{2 q}(x) \sim c_q (\log x)^{q^2}, \]
	where $c_q$ is an explicit constant that splits into an arithmetic and a geometric part, and where the arithmetic part coincides with the one that is conjectured for the integral moments of the Riemann zeta function.
	The case of non-integral values of $q$ was first considered by Bondarenko, Heap and Seip in \cite{BHS1}. Their result directly implies that for fixed $q > 0$ and $\alpha \ge 1$ we have
	\[ (\log x)^{(q \alpha)^2} \ll \Psi_{2q,\alpha}(x) \ll
	\begin{cases}
	(\log x)^{(q \alpha)^2}, &\mbox{ if } q > \frac{1}{2} \\
	(\log x)^{\alpha^2/4} \log \log x , &\mbox{ if } q = \frac{1}{2} \\
	(\log x)^{q \alpha^2/2} &\mbox{ if } q < \frac{1}{2}
	\end{cases}.
	\]
	For $\alpha = 1$ and $q \le \frac{1}{2}$, the upper bound was improved by Heap \cite{Heap1}, who established that
	\[ \Psi_{2q}(x) \ll (\log x)^{\alpha_q} (\log \log x)^{\frac{1}{2} - \alpha_q} \]
	with $\alpha_q = \frac{q}{4(1-q)}$.
	
	A rather natural guess to make at this point concerning the size of these pseudomoments would be that for all $\alpha \ge 1$ and $q > 0$ the right order of magnitude should be $(\log x)^{(q \alpha)^2}$. And not only because this holds for $q > \frac{1}{2}$, but also since one might perhaps expect the $2 q$-th pseudomoment of the $\alpha$-th power of $\zeta$ to correspond to the $2 q \alpha$-th moment $M_{2 q \alpha}(T)$ of $\zeta$, where
	\[ M_{2 k}(T) := \frac{1}{T} \int_T^{2 T} \left| \zeta \left( \frac{1}{2} + i t \right)  \right|^{2 k} \, dt. \]
	It is known unconditionally due to Heap, Radziwi{\l}{\l} and Soundararajan \cite{HRS1} that $M_{2 k}(T) \ll (\log T)^{k^2}$ when $0 < k \le 2$ , and due to Radziwi{\l}{\l} and Soundararajan \cite{RadziwillSound1} that $M_{2 k}(T) \gg (\log T)^{k^2}$ when $k \ge 1$. In fact, under the Riemann Hypothesis it is known that $M_{2 k}(T) \asymp (\log T)^{k^2}$ for all $k > 0$, the lower bound due to Ramachandra \cite{Ramachandra1,Ramachandra2} and Heath-Brown \cite{HeathBrown1} and the upper bound due to Harper \cite{Harper5} building on work of Soundararajan \cite{Sound1}.
	
	However, it turns out that this order of magnitude can not be correct, at least when $\alpha > 1$ and $q$ is very small depending on $\alpha$. This follows from a result of Bondarenko et al. \cite[Theorem 2]{BBSSZ1}, which implies that
	\[ \Psi_{2 q, \alpha}(x) \gg_\varepsilon (\log x)^{2 q \log \alpha - \varepsilon} \]
	for $q < \frac{1}{2}$ and $\alpha > 1$, hence larger than $(\log x)^{(q \alpha)^2}$ when $q$ is sufficiently small depending on $\alpha$.
	
	\subsection{Statement of Results}
	
	We note that throughout this work, the parameters $q$ and $\alpha$ will be fixed.
	
	\begin{theorem}\label{MainTheorem}
		Suppose that $1 \le \alpha < 2$. Then we have
		\[
		\Psi_{2q, \alpha}(x) \ll
		\begin{cases}
		(\log x)^{2 (\alpha - 1) q} , &\mbox{ if } 0 < q < \frac{2(\alpha - 1)}{\alpha^2} \\
		(\log x)^{(2 (\alpha - 1) / \alpha)^2} \log \log x &\mbox{ if } q = \frac{2(\alpha - 1)}{\alpha^2} \\
		(\log x)^{(q \alpha)^2} , &\mbox{ if } \frac{2(\alpha - 1)}{\alpha^2} < q \le \frac{1}{2} \\
		\end{cases}.
		\]
	\end{theorem}

	In particular, if $\alpha = 1$, this result together with the lower bound of Bondarenko, Heap and Seip for $q > 0$ and their upper bound for $q > \frac{1}{2}$ imply the following
	
	\begin{corollary}
		For all $q > 0$, we have
		\[\Psi_{2 q}(x) \asymp (\log x)^{q^2}. \]
	\end{corollary}
	
	In terms of lower bounds, we prove that
	
	\begin{theorem}\label{LowerBoundThm}
		For all $1 \le \alpha \le 2$ and $0 < q < \frac{2(\alpha-1)}{\alpha^2}$, we have
		\[ \Psi_{2q,\alpha}(x) \gg \frac{(\log x)^{2(\alpha-1)q}}{(\log \log x)^{6 + \alpha q}} \]
		when $x$ is sufficiently large. For all $\alpha > 2$ and $0 < q < \frac{1}{2}$, we have
		\[ \Psi_{2q,\alpha}(x) \gg \frac{(\log x)^{q \alpha^2/2}}{(\log \log x)^{2 \alpha^2 / 3 + 5 + q \alpha}} \]
		when $x$ is sufficiently large.
	\end{theorem}
	
	Combining this with the results of Bondarenko, Heap and Seip \cite{BHS1}, we get the following table. The column labelled exponent of $\log x$ refers to the correct exponent of $\log x$ for the size of the pseudomoments on the given range of $\alpha$ and $q$, thus ignoring all factors $(\log \log x)^{O(1)}$. The column labelled $\theta$ refers to the value of $\theta$ such that the main contribution in terms of moments of integrals of random Euler products comes from around $(\log x)^{\theta}$, and might become more clear after reading the heuristic discussion that follows.
	
	\vspace{10pt}
	\begin{center}
		{\renewcommand{\arraystretch}{1.5}
		\begin{tabular}{ | c | c | c | c | c | c |}
			\hline
			range of $\alpha$ & range of $q$ & exponent of $\log x$ & $\theta$ & upper bound & lower bound \\ \hline \hline
			$1 \le \alpha < 2$ & $0 < q \le \frac{2(\alpha-1)}{\alpha^2}$ & $2(\alpha-1)q$ & 0 & Theorem \ref{MainTheorem} & Theorem \ref{LowerBoundThm} \\
			\hline
			$1 \le \alpha < 2$ & $\frac{2(\alpha-1)}{\alpha^2} < q < \frac{1}{2}$ & $(q \alpha)^2$ & $-1$ & Theorem \ref{MainTheorem} & \cite{BHS1} \\
			\hline
			$\alpha \ge 2$ & $0 < q < \frac{1}{2}$ & $\frac{q \alpha^2}{2}$ & $\frac{\alpha^2}{4} - 1$ & \cite{BHS1} & Theorem \ref{LowerBoundThm} \\ \hline
			$\alpha \ge 1$ & $ q \ge \frac{1}{2}$ & $(q \alpha)^2$ &  & \cite{BHS1} & \cite{BHS1} \\ \hline
		\end{tabular}
	}
	\end{center}

	\vspace{10pt}
	
	\subsection{Proof strategy and heuristic discussion}
	
	Note that our methods are of a rather different nature to those of \cite{BHS1}, who apply ideas of a more functional-analytic flavour. They provide general inequalities in the Hardy space of Dirichlet series, where the norm is defined in such a way that, when applied to the divisor functions, one obtains the pseudomoments (up to normalisation).
	
	On the other hand, we follow a more probabilistic approach along the lines of work of Harper \cite{Harper3,Harper1}. We will now try to explain the basic strategy of the proof, and give a heuristic argument why one should expect Theorems \ref{MainTheorem} and \ref{LowerBoundThm} to hold. 
	
	A random completely multiplicative function $f : \mathbb{N} \to \mathbb{C}$ whose values at the primes are independent and uniformly distributed on the complex unit circle will be called a \textit{Steinhaus random multiplicative function}.
	
	The first step is to note that by the Bohr correspondence (see e.g. \cite[Section 3]{SaksmanSeip1}), we have
	\[ \Psi_{2 q, \alpha}(x) = \mathbb{E} \Big[ \, \Big| \sum_{n \le x} \frac{d_\alpha(n) f(n)}{\sqrt{n}} \Big|^{2 q} \, \Big]. \]
	
	One then shows that, roughly speaking, we have
	\[ \Psi_{2q, \alpha}(x) \approx \frac{1}{(\log x)^q} \mathbb{E} \bigg[ \Big( \int_{\mathbb{R}} \frac{|F(1/2+it)|^{2 \alpha}}{|\frac{1}{\log x} + i t|^2}\, dt \Big)^q \bigg], \tag{$\star$} \]
	where
	\[ F(s) = \prod_{p \le x} \left( 1 - \frac{f(p)}{p^s} \right)^{-1} \]
	is the Euler product associated to $f$. This step is somewhat similar to \cite[Proposition 1 and 3]{Harper1}, although especially for the upper bound, adaptations have to be made. We remark that proving this in fact constitutes a rather significant part of this work, namely all of section \ref{REPSection}. The (random) Euler product has the big advantage over the random sum in that it is a product of independent random variables and thus much more tractable to probabilistic methods, whereas the initial sum had a rather intricate dependency structure. We can then divide the integral into dyadic ranges and are thus left with the task of bounding expressions of the type
	\[ \mathbb{E} \bigg[ \Big( \int_T^{2 T} |F(1/2+it)|^{2 \alpha} \, dt \Big)^q \bigg] \]
	with $T \gg \frac{1}{\log x}$. Note that in $(\star)$ we also get a term coming from $|t| \le \frac{1}{\log x}$, but using that $F$ is translation-invariant in law, one can verify that it gives the same contribution as $T = \frac{2}{\log x}$. Now since
	\[ F(1/2+it) = \prod_{p \le x} \left( 1 - \frac{f(p) e^{-i t \log p}}{\sqrt{p}} \right)^{-1}, \]
	what happens is that the Euler factor is roughly constant over the whole range $[T,2 T]$ precisely when $ T\log p \ll 1$. However, this can only happen when $T \ll 1$. Thus, the problem naturally splits into the ranges $\frac{1}{\log x} \ll T \ll 1$ and $T \gg 1$. 
	
	On the first range, the contribution to the Euler product is roughly constant on the range $p \le \exp(1/T)$, so on this range of $T$ it turns out to be useful to split the Euler product into the ``small" primes, which satisfy this condition, and the ``large" primes $p > \exp(1/T)$. On the range $T \gg 1$ on the other hand there are only large primes in the sense that no Euler product factors are expected to be constant over the whole range of integration.
	
	These ideas strongly resemble observations made in \cite{AOR1}. There, the authors analyse the behaviour of 
	\[ \int_{|h| < (\log H)^\theta} \left|\zeta \left( \frac{1}{2} + i t + i h \right) \right|^{2 q} \, dh \]
	for $q > 0,\, \theta > -1$ and "most" values of $t$. There, it turns out that the problem also naturally splits into $- 1 < \theta < 0$ and $\theta > 0$.
	In fact, the resemblance goes further - in their work, it also turns out to be useful to split into small and large primes in a very similar way to our strategy here. Their basic idea is to relate $\zeta$ on random intervals to random quantities that are quite similar to the ones studied in this work.
	
	Now suppose that $\frac{1}{\log x} \ll T \ll 1$. Let $F^s$ be the Euler product over the small primes ($\le \exp(1/T)$) and $F^l$ the Euler product over the large primes. Heuristically, because the Euler product over the small primes is roughly constant and since the values $f(p)$ at different primes are independent, we should be able to pull the small prime contribution out to deduce that
	\[ \mathbb{E} \bigg[ \Big( \int_T^{2 T}|F(1/2+it)|^{2 \alpha}\, dt \Big)^q \bigg] \approx \mathbb{E} [ \, |F^s(1/2)|^{2 \alpha q } \,] \mathbb{E} \bigg[ \Big( \int_T^{2 T} |F^l(1/2+it)|^{2 \alpha}\, dt \Big)^q \bigg].\]
	Moreover, an elementary computation shows that
	\[\mathbb{E} [ \, |F^s(1/2)|^{2 \alpha q } \,] \asymp T^{-(q \alpha)^2}. \]
	Since $e^{-i t \log p}$ roughly changes on a scale of $\frac{1}{\log x}$, at least when $p$ is fairly large, one might guess that
	\[ \int_T^{2 T}|F^l(1/2+it)|^{2 \alpha} \, dt \approx \frac{1}{\log x} \sum_{T \log x < n \le 2 T \log x} \left|F^l \left(\frac{1}{2} + i \frac{n}{\log x} \right) \right|^{2 \alpha}.  \]
	Next, note that by Taylor expansion we have
	\[ \log |F^l(1/2 + it)| \approx \sum_{\exp(1/T) < p \le x} \frac{\Re (f(p) p^{-it})}{\sqrt{p}}.  \]
	This is a sum of independent random variables whose individual contributions are not too large, so that one might expect this to be roughly a Gaussian with mean $0$ and variance
	\[ \sum_{\exp(1/T) < p \le x} \frac{\mathbb{E}[\Re(f(p) p^{-it})^2]}{p} \approx \frac{1}{2}(\log \log x + \log T). \]
	Thus, one might model $\int_T^{2 T}|F^l(1/2+it)|^{2 \alpha} \, dt$ by the sum of $T \log x$ random variables given by the exponentials of Gaussians with mean $0$ and variance $(\log \log x + \log T)/2$. One might hope that these random variables are not too correlated and thus replace them by independent ones. This is perhaps the most unclear step in this heuristic argument, and likely leads to correction factors of size $(\log \log x)^{O(1)}$, as is featured in the more-than-squareroot cancellation observed in \cite{Harper1}. Elementary probabilistic calculations show that the sum that arises is typically dominated by the largest summand. Since the maximum of $n$ independent Gaussians with mean $0$ and variance $\sigma^2$ is $\approx \sigma \sqrt{2 \log n}$ with high probability, we obtain that
	\[ \max_{T \log x < n \le 2 T \log x} \left|F^l \left(\frac{1}{2} + i \frac{n}{\log x} \right) \right| \approx \exp \left(\sqrt{2 (\log \log x + \log T)} \sqrt{\frac{1}{2} (\log \log x + \log T)} \right) = T \log x. \]
	Putting these heuristic estimates together, one concludes that
	\[ \mathbb{E} \bigg[ \Big( \int_T^{2 T} |F^l(1/2+it)|^{2 \alpha}\, dt \Big)^q \bigg] \approx \frac{(T \log x)^{2 \alpha q}}{(\log x)^q}. \]
	The contribution of $t \in [T,2T]$ to ($\star$) is thus
	\[ \approx (\log x)^{-q} T^{-2 q} T^{-(q \alpha)^2} T^{2 \alpha q} (\log x)^{2 \alpha q - q} = T^{2 (\alpha - 1) q - (q \alpha)^2 } (\log x)^{2 (\alpha - 1) q} \]
	when $\frac{1}{\log x} \ll T \ll 1$. Note that the exponent of $T$ is positive iff $q < \frac{2(\alpha - 1)}{\alpha^2}$. Plugging in the respective extremes of the range give a contribution of roughly
	\[\begin{cases}
		(\log x)^{2(\alpha - 1)q } &\mbox{ if } 0 < q \le \frac{2(\alpha - 1)}{\alpha ^2} \\
		(\log x)^{(q \alpha)^2} &\mbox{ if } \frac{2(\alpha - 1)}{\alpha^2} < q < \frac{1}{\alpha}
	\end{cases}\]
	to the pseudomoments. The reader might want to compare this to Theorem \ref{MainTheorem}.
	
	One can employ the same heuristic for $T \gg 1$. In this case, there are only large primes, so the argument simplifies somewhat. We therefore have $T \log x$ random variables that are roughly the exponential of a Gaussian with mean $0$ and variance $\frac{1}{2} \log \log x$, so that the maximum should roughly be
	\[ \exp \left( \sqrt{2(\log \log x + \log T)} \sqrt{\frac{1}{2} \log \log x} \right) = \exp(\sqrt{(\log \log x)^2 + \log \log x \log T}). \]
	On this range, the contribution of $t \in [T, 2T]$ to ($\star$) is hence
	\[ \approx T^{-2 q} (\log x)^{-2 q} \exp \left( 2 \alpha q \sqrt{(\log \log x)^2 + \log \log x \log T}\right) \]
	when $0 < q < \frac{1}{\alpha}$. Maximizing this in terms of $T$ on the range $T \gg 1$ gives a contribution of 
	\[  
	\begin{cases}
	(\log x)^{q \alpha^2/2}, &\mbox{ if } \alpha > 2 \text{ at } T = (\log x)^{\alpha^2/4-1} \\
	(\log x)^{2 (\alpha - 1) q}, &\mbox{ if } 1 \le \alpha \le 2 \text{ at } T = 1
	\end{cases}.
	\]
	This suggests that for $\alpha > 2$ the upper bound of Bondarenko, Heap and Seip should in fact be roughly the correct answer, as is confirmed by Theorem \ref{LowerBoundThm}.
	
	\begin{acknow}
		I would like to thank Adam Harper for suggesting this problem to me, and for numerous invaluable discussions and comments. I would also like to thank the University of Warwick for their hospitality during my stay in the summer of 2019.
	\end{acknow}
	
	\section{Reduction to moments of integrals of random Euler products}\label{REPSection}
	
	We begin by recording the following version of Plancherel's identity for Dirichlet series (see e.g. \cite[(5.26)]{MontVaughan}).
	
	\begin{lemma}\label{Plancherel}
		Let $(a_n)_{n \ge 1}$ be a sequence of complex numbers, and let $A(s) := \sum_{n \ge 1} \frac{a_n}{n^s}$ be the corresponding Dirichlet series with abscissa of convergence $\sigma_c$. Then for any $\sigma > \max \{ 0, \sigma_c \}$, we have
		\[ \int_0^\infty \frac{\big| \sum_{n \le x} a_n \big|^2}{x^{1+ 2 \sigma}} \, dx = \frac{1}{2 \pi} \int_{\mathbb{R}} \left| \frac{A(\sigma + i t)}{\sigma+i t} \right|^2 \, dt. \]
	\end{lemma}

	Furthermore, we will record the following bound from \cite[Number Theory Result 1]{Harper1}, compare also Lau, Tenenbaum and Wu \cite[Lemma 2.1]{LTW1}. We write $\Omega(n)$ for the number of prime factors of $n$ counted with multiplicity.
	
	\begin{lemma}\label{NTR1}
		Let $0 < \delta < 1$ and $\alpha \ge 1$. Suppose that $\max \{ 3 ,2 \alpha \} \le y \le z \le y^{10}$ (say) and $1 < u \le v(1-y^{- \delta})$. Then we have
		\[ \sum_{\substack{ u \le n \le v \\ p \, | \, n \Rightarrow y < p \le z }} \alpha^{\Omega(n)} \ll_\delta \frac{(v-u) \alpha}{\log y} \prod_{y < p \le z} \left( 1 - \frac{\alpha}{p} \right)^{-1}. \] 
	\end{lemma}

	\subsection{Upper bounds}
	
	The main part of this section will be devoted to deducing an analogue of \cite[Proposition 1]{Harper1}, giving an upper bound for the pseudomomoments in terms of moments of integrals of random Euler products. In order to state this properly, we introduce the notation
	\[ F_k(s) := \prod_{p \le x^{e^{-k}}} \left(1 - \frac{f(p)}{p^s} \right)^{-1}. \]
	Moreover, for a given integer $n$ we denote by $P(n)$ the largest prime divisor of $n$. We will frequently be using that for $\alpha \ge 1$ the divisor functions satisfy the inequality $d_\alpha^2 \le d_{\alpha^2}$. To see this, note first that from multiplicativity it suffices to verify this at prime powers. From there, it is an elementary induction exercise using standard properties of binomial coefficients, noting that $d_\alpha(p^j) = \binom{j + \alpha - 1}{j}$. We leave the rest to the reader.
	
	\begin{proposition}\label{REP}
	Let $\alpha \ge 1$ and $0 < q \le \frac{1}{2}$ be fixed and let $K = [\log \log \log x]$.
	Then we have
	\begin{align*}
	\Psi_{2q,\alpha}(x) \ll \frac{1}{(\log x)^q} &\sum_{0 \le k \le K} \mathbb{E} \Bigg[ \bigg( \int_1^{x^{1-e^{-(k+1)}}} \bigg| \sum_{\substack{n > z \\ P(n)  \le x^{e^{-(k+1)}} }} \frac{d_\alpha(n) f(n)}{\sqrt{n}}  \bigg|^2 \, \frac{dz}{z^{1-2 k / \log x}} \bigg)^q \bigg] \\ + &\sum_{0 \le k \le K} e^{- e^k q} \mathbb{E} \left[ |F_k(1/2)|^{2 \alpha q} \right] + 1.
	\end{align*}
	\end{proposition}

	One rather crucial difference compared to \cite[Proposition 1]{Harper1} here is that we are summing over $n > z$ instead of $n \le z$ inside the integral. This is very helpful in order to achieve uniformity over $k$. One exemplary reason for that is the fact that
	\[ \sum_{\substack{ n \le x \\ P(n) \le x^{e^{-(k+1)}} }} \frac{d_\alpha(n)^2}{n} \quad \text{and} \quad \sum_{ n \le x } \frac{d_\alpha(n)^2}{n} \] are fairly comparable in size when $k$ is close to $K$, whereas
	\[ \sum_{\substack{ n > x \\ P(n) \le x^{e^{-(k+1)}} }} \frac{d_\alpha(n)^2}{n} \]
	(an expression that will appear in the proof) is quite significantly smaller. Note that the corresponding expressions
	\[ \sum_{\substack{ n \le x \\ P(n) \le x^{e^{-(k+1)}} }} 1 \quad \text{and} \quad \sum_{ n \le x } 1 \]
	in \cite[Proposition 1]{Harper1} on the other hand are quite far apart in size. Similar features appear in other error bounds when $k$ is close to $K$. Note also that in any case the only reason we are allowed to switch to $n > x$ is that the complete sums over $n$ with $P(n) \le x$, which are just the full Euler products over $p \le x$, are not too large here (as will be illustrated in the proof), which is not the case in the work there.

	\begin{proof}[Proof of Proposition \ref{REP}]
		The Bohr correspondence
		tells us that
		\[ \Psi_{2 q, \alpha}(x) = \Big\Vert \sum_{n \le x} \frac{d_\alpha(n) f(n)}{\sqrt{n}} \Big\Vert_{2 q}^{2q}, \]
		where $\Vert \cdot \Vert_{2 q} = \mathbb{E}[ |\cdot|^{2 q}]^{1/{2 q}}$. Note that this does not define a norm when $q < 1/2$, but only a pseudonorm (but we might still sometimes refer to it as a norm).
		
		Now
		\begin{align*}\allowdisplaybreaks
		\Big\Vert \sum_{n \le x} \frac{d_\alpha(n) f(n)}{\sqrt{n}} \Big\Vert_{2 q}^{2q}
		&\le \Big\Vert \sum_{P(n) \le x} \frac{d_\alpha(n) f(n)}{\sqrt{n}} \Big\Vert_{2 q}^{2q} + \Big\Vert \sum_{\substack{n > x \\ P(n) \le x}} \frac{d_\alpha(n) f(n)}{\sqrt{n}} \Big\Vert_{2 q}^{2q} \\
		&= \Big \Vert \prod_{p \le x} \left( 1 + \frac{d_\alpha(p) f(p)}{p^{1/2}} + \frac{d_\alpha(p^2) f(p)^2}{p} + \dots \right)\Big \Vert_{2 q}^{2 q} + \Big\Vert \sum_{\substack{n > x \\ P(n) \le x}} \frac{d_\alpha(n) f(n)}{\sqrt{n}} \Big\Vert_{2 q}^{2q}  \\
		&= \Big \Vert \prod_{p \le x} \left( 1 - \frac{f(p)}{p^{1/2}}\right)^{-\alpha} \Big \Vert_{2 q}^{2 q} + \Big\Vert \sum_{\substack{n > x \\ P(n) \le x}} \frac{d_\alpha(n) f(n)}{\sqrt{n}} \Big\Vert_{2 q}^{2q} \\
		&= \mathbb{E} \left[ |F(1/2)|^{2 \alpha q} \right] + \Big\Vert \sum_{\substack{n > x \\ P(n) \le x}} \frac{d_\alpha(n) f(n)}{\sqrt{n}} \Big\Vert_{2 q}^{2q}.
		\end{align*}
		
		We can then subdivide the sum according to the size of the largest prime factor, to obtain that
		\begin{align*} \Big\Vert \sum_{n \le x} \frac{d_\alpha(n) f(n)}{\sqrt{n}} \Big\Vert_{2 q}^{2q} \le \mathbb{E} \left[ |F(1/2)|^{2 \alpha q} \right] + \sum_{0 \le k \le K} &\Big\Vert \sum_{\substack{n > x \\ x^{e^{-(k+1)} < P(n) \le x^{e^{-k}} } }} \frac{d_\alpha(n) f(n)}{\sqrt{n}} \Big\Vert_{2 q}^{2 q} \\ + &\Big \Vert \sum_{\substack{n > x \\ P(n) \le x^{e^{-(K+1)}}  }} \frac{d_\alpha(n) f(n)}{\sqrt{n}} \Big \Vert_{2 q}^{2 q}. 
		\end{align*}
		In order to bound the last term, we can trivially bound the $2q$-norm by the $2$-norm and use orthogonality to deduce that 
		\[ \Big \Vert \sum_{\substack{n > x \\ P(n) \le x^{e^{-(K+1)}}}} \frac{d_\alpha(n) f(n)}{\sqrt{n}} \Big \Vert_{2 q}^{2 q} \le \bigg(\sum_{\substack{n > x \\ P(n) \le x^{e^{-(K+1)}}}} \frac{d_\alpha(n)^2}{n}\bigg)^q.  \]
		But this can be dealt with by means of Rankin's trick: For any constant $C>0$, we have
		\begin{align*}
		\sum_{\substack{n > x \\ P(n) \le y  }} \frac{d_\alpha(n)^2}{n} &\le x^{-C/\log y} \sum_{\substack{n > x \\ P(n) \le y }} \frac{d_\alpha(n)^2}{n^{1-C/\log y}} \le x^{-C/\log y} \sum_{P(n) \le y} \frac{d_\alpha(n)^2}{n^{1-C/\log y}} \\ &\ll x^{-C/\log y} \prod_{p \le y} \left(1 - \frac{1}{p^{1-C/\log y}} \right)^{- \alpha ^ 2} \ll x^{-C/\log y} (\log y)^{\alpha^2}
		\end{align*}
		(using that $d_\alpha(n)^2 \le d_{\alpha^2}(n)$ in the third step). Taking $y = x^{1/\log \log x}$ and $C=\alpha^2$ and using that $K = [\log \log \log x]$ thus gives
		\[ \sum_{\substack{n > x \\ P(n) \le x^{e^{-(K+1)}}  }} \frac{d_\alpha(n)^2}{n} \ll (\log x)^{-\alpha^2} \left( \frac{\log x}{\log \log x} \right)^{\alpha^2} \ll 1. \]
		Putting the bounds up to this point together tells us that
		\begin{equation} \Psi_{2q,\alpha}(x) \ll \sum_{0 \le k \le K} \Big\Vert \sum_{\substack{n > x \\ x^{e^{-(k+1)} < P(n) < x^{e^{-k}} } }} \frac{d_\alpha(n) f(n)}{\sqrt{n}} \Big\Vert_{2 q}^{2 q} +  \mathbb{E} \left[ |F(1/2)|^{2 \alpha q} \right] + 1. \end{equation}
		
		Next, let $\mathbb{E}^{(k)}$ denote the conditional expectation given $(f(p))_{p \le x^{e^{-(k+1)}}}$. Using H\"{o}lder's inequality for conditional expectations as well as the independence of $f(p)$ at different primes and orthogonality (compare \cite[Proposition 1]{Harper1}), we have
		\begin{align*}\allowdisplaybreaks
		&\quad \; \bigg\Vert \sum_{\substack{n > x \\ x^{e^{-(k+1)} < P(n) \le x^{e^{-k}} } }} \frac{d_\alpha(n) f(n)}{\sqrt{n}} \bigg\Vert_{2 q}^{2 q} \\
		&= \bigg \Vert \sum_{\substack{ m > 1 \\ p \, | \, m \Rightarrow x^{e^{-(k+1)}} < p \le x^{e^{-k}}}} \frac{d_\alpha(m) f(m)}{\sqrt{m}} \sum_{\substack{ n > x/m \\ P(n)\le x^{e^{-(k+1)}} }} \frac{d_\alpha(n) f(n)}{\sqrt{n}} \bigg \Vert_{2 q}^{2 q} \\
		&= \mathbb{E} \Bigg[ \mathbb{E}^{(k)} \Bigg[ \; \bigg| \sum_{\substack{m > 1 \\ p \, | \, m \Rightarrow x^{e^{-(k+1)}} < p \le x^{e^{-k}} }} \frac{d_\alpha(m) f(m)}{\sqrt{m}} \sum_{\substack{ n > x/m \\ P(n)\le x^{e^{-(k+1)}} }} \frac{d_\alpha(n) f(n)}{\sqrt{n}} \bigg|^{2 q} \; \Bigg] \Bigg] \\
		&\le \mathbb{E} \Bigg[ \Bigg( \mathbb{E}^{(k)} \Bigg[ \bigg|\sum_{\substack{ m > 1 \\ p \, | \, m \Rightarrow x^{e^{-(k+1)}} < p \le x^{e^{-k}} }} \frac{d_\alpha(m) f(m)}{\sqrt{m}} \sum_{\substack{ n > x/m \\ P(n)\le x^{e^{-(k+1)}} }} \frac{d_\alpha(n) f(n)}{\sqrt{n}}  \bigg|^2 \Bigg] \Bigg)^q \Bigg] \\
		&= \bigg \Vert \sum_{\substack{m > 1 \\ p \, | \, m \Rightarrow x^{e^{-(k+1)}} < p \le x^{e^{-k}} }} \frac{d_\alpha(m)^2}{m} \; \bigg| \sum_{\substack{ n > x/m \\ P(n)\le x^{e^{-(k+1)}} }} \frac{d_\alpha(n) f(n)}{\sqrt{n}} \bigg|^2 \, \bigg \Vert_q^q.
		\end{align*}
		The next step is to smoothen the inner sum. Again we proceed in a very similar fashion to \cite[Proposition 1]{Harper1}, setting (say) $X=\exp(\sqrt{\log x})$ and noting that
		\begin{align}
		&\bigg \Vert \sum_{\substack{m > 1 \\ p \, | \, m \Rightarrow x^{e^{-(k+1)}} < p \le x^{e^{-k}} }} \frac{d_\alpha(m)^2}{m} \; \bigg| \sum_{\substack{ n > x/m \\ P(n)\le x^{e^{-(k+1)}} }} \frac{d_\alpha(n) f(n)}{\sqrt{n}} \bigg|^2 \, \bigg \Vert_q^q \nonumber \\
		\ll \, &\bigg \Vert \sum_{\substack{m > 1 \\ p \, | \, m \Rightarrow x^{e^{-(k+1)}} < p \le x^{e^{-k}} }} \frac{X d_\alpha(m)^2}{m^2} \int_{m(1-1/X)}^m \bigg| \sum_{\substack{ n > x/t \\ P(n)\le x^{e^{-(k+1)}} }} \frac{d_\alpha(n) f(n)}{\sqrt{n}} \bigg|^2 \, dt \, \bigg \Vert_q^q \nonumber \\
		+ \, &\bigg \Vert \sum_{\substack{m > 1 \\ p \, | \, m \Rightarrow x^{e^{-(k+1)}} < p \le x^{e^{-k}} }} \frac{X d_\alpha(m)^2}{m^2} \int_{m(1-1/X)}^m \bigg| \sum_{\substack{ x/m < n \le x/t \\ P(n)\le x^{e^{-(k+1)}} }} \frac{d_\alpha(n) f(n)}{\sqrt{n}} \bigg|^2 \, dt \, \bigg \Vert_q^q. \label{SmoothingError}
		\end{align}
		The range of summation for the inner sum in the second term is rather small, so we might expect this to only give a minor contribution. Indeed, trivially bounding the $q$-norm by the $1$-norm, pulling the expectation inside and then using orthogonality, the second term in (\ref{SmoothingError}) is
		\begin{align*}\allowdisplaybreaks
		&\le \Bigg( \sum_{\substack{m > 1 \\ p \, | \, m \Rightarrow x^{e^{-(k+1)}} < p \le x^{e^{-k}} }} \frac{X d_\alpha(m)^2}{m^2} \int_{m(1-1/X)}^m \mathbb{E} \bigg[ \; \bigg| \sum_{\substack{ x/m < n \le x/t \\ P(n)\le x^{e^{-(k+1)}} }} \frac{d_\alpha(n) f(n)}{\sqrt{n}} \bigg|^2 \, \bigg] \, dt  \Bigg)^q \\
		&\le \Bigg( \sum_{\substack{ m > 1\\ p \, | \, m \Rightarrow x^{e^{-(k+1)}} < p \le x^{e^{-k}} }} \frac{d_\alpha(m)^2}{m} \sum_{\substack{ \frac{x}{m} < n \le \frac{x}{m(1-1/X)} \\ P(n)\le x^{e^{-(k+1)}} }} \frac{d_\alpha(n)^2}{n} \Bigg)^q \\
		&\ll \Bigg( \frac{1}{x} \sum_{\substack{m > 1 \\ p \, | \, m \Rightarrow x^{e^{-(k+1)}} < p \le x^{e^{-k}} }} d_\alpha(m)^2 \sum_{\substack{ \frac{x}{m} < n \le \frac{x}{m(1-1/X)} \\ P(n)\le x^{e^{-(k+1)}} }} d_\alpha(n)^2 \Bigg)^q.
		\end{align*}
		In order to bound this, we remark that by a result of Shiu \cite{Shiu1}, we have
		\[ \sum_{x-y < n \le x} d_a(n) \ll y (\log x)^{a - 1} \]
		and
		\[ \sum_{\substack{x-y < n \le x \\ p \, | \, n \Rightarrow x^{e^{-(k+1)}} < p \le x^{e^{-k}} }} d_a(n) \ll y \]
		for fixed $a$ and uniformly over $x^{1/3} \le y \le x$, say.
		We now use a hyperbola-type argument, subdividing the first sum into the range $1 < m \le \sqrt{x}$ and $m >  \sqrt{x}$. We then interchange the sum on the latter range, and thus have
		\begin{align*}
		&\frac{1}{x} \sum_{\substack{m > 1 \\ p \, | \, m \Rightarrow x^{e^{-(k+1)}} < p \le x^{e^{-k}} }} d_\alpha(m)^2 \sum_{\substack{ \frac{x}{m} < n \le \frac{x}{m(1-1/X)} \\ n \text{ is } x^{e^{-(k+1)}} \text{-smooth} }} d_\alpha(n)^2 \\
		\le &\frac{1}{x} \sum_{\substack{1 < m \le \sqrt{x} \\ p \, | \, m \Rightarrow x^{e^{-(k+1)}} < p \le x^{e^{-k}} }} d_\alpha(m)^2 \sum_{\frac{x}{m} < n \le \frac{x}{m(1-1/X)}} d_{\alpha^2}(n) 
		\\ + &\frac{1}{x} \sum_{\substack{n \le \frac{\sqrt{x}}{1-1/X} \\ P(n) \le x^{e^{-(k+1)}} }} d_\alpha(n)^2 \sum_{\substack{ \frac{x}{n} < m \le \frac{x}{n(1-1/X)} \\ p \, | \, m \Rightarrow x^{e^{-(k+1)}} < m \le x^{e^{-k}} }} d_{\alpha^2}(m) \\ 
		\ll &\frac{(\log x)^{\alpha^2 - 1}}{X} \sum_{\substack{m > 1 \\ p \, | \, m \Rightarrow x^{e^{-(k+1)}} < m \le x^{e^{-k}} }} \frac{d_\alpha(m)^2}{m} + \frac{1}{X} \sum_{\substack{n \le \frac{\sqrt{x}}{1-1/X} \\ P(n) \le x^{e^{-(k+1)}} }} \frac{d_\alpha(n)^2}{n} \\ 
		\ll &\frac{(\log x)^{\alpha^2 - 1}}{X} \prod_{x^{e^{-(k+1)}} < p \le x^{e^{-k}}} \left( 1 - \frac{1}{p}  \right)^{-\alpha^2} + \frac{1}{X} \prod_{p \le x^{e^{-(k+1)}}} \left( 1 - \frac{1}{p} \right)^{- \alpha^2} \\
		\ll &\frac{(\log x)^{\alpha^2}}{X},
		\end{align*} 
		
		which is easily $\ll 1$ when summed over $0 \le k \le K$ (after taking the $q$-th power).
		
		Regarding the first term in (\ref{SmoothingError}), we can interchange sum and integral to arrive at
		\begin{equation}\label{tIntegral} \bigg \Vert \int_{x^{e^{-(k+1)}}}^\infty \bigg| \sum_{\substack{n > x/t \\ P(n) \le x^{e^{-(k+1)}} }} \frac{d_\alpha(n) f(n)}{\sqrt{n}}  \bigg|^2 \sum_{\substack{t < m \le t/(1-1/X) \\ p \, | \, m \Rightarrow x^{e^{-(k+1)}} < p \le x^{e^{-k}} }} \frac{X d_\alpha(m)^2 }{m^2} \, dt \bigg \Vert_q^q. \end{equation}
		
		Concluding our estimates so far, we have now proven that
		\begin{align*}
		\Psi_{2 q, \alpha}(x) &\ll \sum_{0 \le k \le K} \bigg \Vert \int_{x^{e^{-(k+1)}}}^\infty \bigg| \sum_{\substack{n > x/t \\ P(n) \le x^{e^{-(k+1)}} }} \frac{d_\alpha(n) f(n)}{\sqrt{n}}  \bigg|^2 \sum_{\substack{t < m \le t/(1-1/X) \\ p \, | \, m \Rightarrow x^{e^{-(k+1)}} < p \le x^{e^{-k}} }} \frac{X d_\alpha(m)^2 }{m^2} \, dt \bigg \Vert_q^q \\ &+ \mathbb{E} \left[ |F(1/2)|^{2 \alpha q} \right] + 1.
		\end{align*}
		
		For the inner sum in (\ref{tIntegral}), note that $d_\alpha(n)^2 \le \alpha^{2 \Omega(n)}$, and that $m > t$ and $p \, | \, m \Rightarrow x^{e^{-(k+1)}} < p \le x^{e^{-k}}$ imply that $m$ has $\ge \frac{e^k \log t}{\log x}$ prime divisors. Thus, Lemma \ref{NTR1} implies
		\begin{align*}
		\sum_{\substack{t < m \le t/(1-1/X) \\ p \, | \, m \Rightarrow x^{e^{-(k+1)}} < p \le x^{e^{-k}} }} \frac{X d_\alpha(m)^2 }{m^2} &\le \frac{X}{t^2} \sum_{\substack{t < m \le t/(1-1/X) \\ p \, | \, m \Rightarrow x^{e^{-(k+1)}} < p \le x^{e^{-k}} }} \alpha^{2 \Omega(m)} \\
		&\le \frac{X}{t^2} e^{-\frac{e^k \log t}{\log x}} \sum_{\substack{t < m \le t/(1-1/X) \\ p \, | \, m \Rightarrow x^{e^{-(k+1)}} < p \le x^{e^{-k}} }} (e \alpha^2)^{\Omega(m)} \\
		&\ll \frac{e^k e^{- \frac{e^k \log t}{\log x}}}{t \log x} \ll \frac{1}{t \log t}.
		\end{align*}
		Subdividing the range of integration in (\ref{tIntegral}) into $t \le x$ and $t > x$, we thus upper-bound it by
		
		\begin{align*}
		&\quad \; \bigg \Vert \int_{x^{e^{-(k+1)}}}^x \bigg| \sum_{\substack{n > x/t \\ n \text{ is } x^{e^{-(k+1)}} \text{-smooth}}} \frac{d_\alpha(n) f(n)}{\sqrt{n}}  \bigg|^2 \, \frac{dt}{t \log t} \bigg \Vert_q^q \\
		&+ \frac{e^{k q}}{(\log x)^q} \bigg \Vert \sum_{\substack{n \ge 1 \\ P(n) \le x^{e^{-(k+1)}} }} \frac{d_\alpha(n) f(n)}{\sqrt{n}}  \bigg \Vert_{2 q}^{2 q} \left( \int_x^\infty \frac{dt}{t^{1+\frac{e^k }{\log x}}} \right)^q \\
		&\ll \bigg \Vert \int_{x^{e^{-(k+1)}}}^x \bigg| \sum_{\substack{n > x/t \\ P(n) \le x^{e^{-(k+1)}}}} \frac{d_\alpha(n) f(n)}{\sqrt{n}}  \bigg|^2 \, \frac{dt}{t \log t} \bigg \Vert_q^q
		+ e^{-e^k q} \mathbb{E}[|F_{k+1} (1/2)|^{2 \alpha q}].
		\end{align*}
		Substituting $z=x/t$, the first term equates to
		\begin{align*} &\bigg \Vert \int_1^{x^{1-e^{-(k+1)}}} \bigg| \sum_{\substack{n > z \\P(n) \le x^{e^{-(k+1)}}}} \frac{d_\alpha(n) f(n)}{\sqrt{n}}  \bigg|^2 \, \frac{dz}{z \log(x/z)} \bigg \Vert_q^q \\
		\ll \frac{1}{(\log x)^q} &\bigg \Vert \int_1^{x^{1-e^{-(k+1)}}} \bigg| \sum_{\substack{n > z \\ P(n) \le x^{e^{-(k+1)}}}} \frac{d_\alpha(n) f(n)}{\sqrt{n}}  \bigg|^2 \, \frac{dz}{z^{1-2 k / \log x}} \bigg \Vert_q^q,
		\end{align*}
		using that $\log(x/z) \gg z^{- 2 k / \log x} \log x$. Putting everything together gives the claim.
	\end{proof}

	\begin{proposition}\label{REP2}
		For any $\alpha \ge 1$, any $0 < q \le \frac{1}{2}$ and any $0 \le k \le K = [\log \log \log x]$, we have
		\begin{align*}
		&\mathbb{E} \Bigg[ \bigg( \int_1^{x^{1-e^{-(k+1)}}} \bigg| \sum_{\substack{n > z \\ P(n) \le x^{e^{-(k+1)}}}} \frac{d_\alpha(n) f(n)}{\sqrt{n}}  \bigg|^2 \, \frac{dz}{z^{1-2 k / \log x}} \bigg)^q \bigg]
		\ll \mathbb{E} \Bigg[ \bigg( \int_{\mathbb{R}} \frac{\left| F_{k+1} \left( \frac{1}{2} - \frac{2 (k+1)}{\log x} + it \right) \right|^{2 \alpha}}{\left| \frac{2 (k+1)}{\log x} + i t \right|^2} \, dt \bigg)^q \Bigg].
		\end{align*}
	\end{proposition}

	We would like to deduce a bound of this type directly from Lemma \ref{Plancherel}, but the $z$-exponent is less than $1$. Thus the idea is to transfer parts of this exponent into the inner sum by means of partial summation, making the $z$-exponent slightly bigger than $1$, and then to apply the Lemma. This causes us to move slightly to the left of the $\frac{1}{2}$-line in terms of $F$, by an amount that should not matter much for the final size of the integrand. This problem does not appear in \cite[Proposition 1]{Harper1} because there, the $z$-exponent is slightly smaller than $2$, thus still far away from the exponent $1$ that limits us in applying Lemma \ref{Plancherel}.
	
	Another issue is that the sum inside the integral now ranges over $n > z$, which is not in the shape of Lemma \ref{Plancherel}, but partial summation also allows us to switch to sums over $n \le z$, assuming that we can deal with the sum over all $n$ with $P(n) \le x^{e^{-(k+1)}}$. But this is again just an Euler product and causes no problems.
	
	\begin{proof}[Proof of Proposition \ref{REP2}.]
		Firstly, we note that
		\begin{align}
		\mathbb{E} &\Bigg[ \bigg( \int_1^{x^{1-e^{-(k+1)}}} \bigg| \sum_{\substack{n > z \\ P(n) \le x^{e^{-(k+1)}}}} \frac{d_\alpha(n) f(n)}{\sqrt{n}}  \bigg|^2 \, \frac{dz}{z^{1-2 k / \log x}} \bigg)^q \bigg] \nonumber \\
		= \lim_{y \to \infty} &\bigg \Vert \int_1^{x^{1-e^{-(k+1)}}} \bigg| \sum_{\substack{z < n \le y \\ P(n) \le x^{e^{-(k+1)}}}} \frac{d_\alpha(n) f(n)}{\sqrt{n}}  \bigg|^2 \, \frac{dz}{z^{1-2 k / \log x}} \bigg \Vert_q^q. \label{limsup}
		\end{align}
		Partial summation applied to the inner sum implies that for any $y > z$ and $\sigma > 0$ we have
		\begin{align}
		&\bigg| \sum_{\substack{z < n \le y \\ P(n) \le x^{e^{-(k+1)}}}} \frac{d_\alpha(n) f(n)}{\sqrt{n}}  \bigg|^2 =  \bigg| \sum_{\substack{z < n \le y \\ P(n) \le x^{e^{-(k+1)}}}} n^{-\sigma} \frac{d_\alpha(n) f(n)}{n^{1/2-\sigma}}  \bigg|^2 \nonumber \\
		\le &\bigg| y^{-\sigma} \sum_{\substack{n \le y \\ P(n) \le x^{e^{-(k+1)}}}} \frac{d_\alpha(n) f(n)}{n^{1/2-\sigma}} \bigg|^2 + \bigg| z^{-\sigma} \sum_{\substack{n \le z \\ P(n) \le x^{e^{-(k+1)}}}} \frac{d_\alpha(n) f(n)}{n^{1/2-\sigma}} \bigg|^2 \nonumber \\
		+ &\bigg| \sigma \int_z^y \sum_{\substack{n \le u \\ P(n) \le x^{e^{-(k+1)}}}} \frac{d_\alpha(n) f(n)}{n^{1/2-\sigma}} \, \frac{du}{u^{1+\sigma}} \bigg|^2. \label{PartialSum}
		\end{align}
		Plugging the first term in (\ref{PartialSum}) into (\ref{limsup}), using that the inner sum does not depend on $z$ and trivially bounding the arising $2 q$-norm by the $2$-norm, gives a contribution
		\begin{align*}
		\le \Big( \int_1^{x^{1-e^{-(k+1)}}} \, \frac{dz}{z^{1-2k/\log x}} \Big)^q \limsup_{y \to \infty} \bigg( y^{-2 \sigma} \sum_{\substack{n \le y \\ P(n) \le x^{e^{-(k+1)}} }} \frac{d_\alpha(n)^2}{n^{1-2 \sigma}} \bigg)^q.
		\end{align*}
		But
		\[ \sum_{\substack{n \le y \\ P(n) \le x^{e^{-(k+1)}} }} \frac{d_\alpha(n)^2}{n^{1-2 \sigma}} \le \sum_{P(n) \le x^{e^{-(k+1)}}} \frac{d_\alpha(n)^2}{n^{1-2 \sigma}} \le \prod_{p \le x^{e^{-(k+1)}}} \left( 1 - \frac{1}{p^{1-2 \sigma}} \right)^{-\alpha^2} \]
		is bounded independent of $y$ (since $\sigma$ will not depend on $y$), hence the contribution vanishes in the limit.
		
		If we plug in the third term of (\ref{PartialSum}) into (\ref{limsup}) with $y$ fixed for now, we arrive at a contribution 
		\[ \bigg \Vert \sigma^2 \int_1^{x^{1-e^{-(k+1)}}} \bigg| \int_z^y \sum_{\substack{n \le u \\ P(n) \le x^{e^{-(k+1)}}}} \frac{d_\alpha(n) f(n)}{n^{1/2-\sigma}} \, \frac{du}{u^{1+\sigma}} \bigg|^2 \, \frac{dz}{z^{1-2 k / \log x}} \bigg \Vert_q^q. \]
		Applying Cauchy-Schwarz to the inner integral and then extending the arising (non-negative) integrals to $\infty$, we see that the last expression is
		\begin{align*} &\le \bigg \Vert \sigma^2 \int_1^{x^{1-e^{-(k+1)}}} \bigg( \int_z^y \Big| \sum_{\substack{n \le u \\ P(n) \le x^{e^{-(k+1)}}}} \frac{d_\alpha(n) f(n)}{n^{1/2-\sigma}} \Big|^2 \, \frac{du}{u^{1+\sigma}} \bigg) \bigg( \int_z^y \, \frac{du}{u^{1+\sigma}} \bigg) \, \frac{dz}{z^{1-2 k / \log x}} \bigg \Vert_q^q \\
		&\le \bigg \Vert \sigma \int_1^{x^{1-e^{-(k+1)}}}  \int_z^\infty \Big| \sum_{\substack{n \le u \\ P(n) \le x^{e^{-(k+1)}}}} \frac{d_\alpha(n) f(n)}{n^{1/2-\sigma}} \Big|^2 \, \frac{du}{u^{1+\sigma}} \, \frac{dz}{z^{1-2 k / \log x + \sigma}} \bigg \Vert_q^q.
		\end{align*}
		Note that the last expression is independent of $y$, so we may take the limit. Lastly, interchanging the two integrals and taking say $\sigma = \frac{4 (k+1)}{\log x}$, we see that this is
		\begin{equation}\label{BeforePl} \ll \bigg \Vert \int_1^\infty \Big| \sum_{\substack{n \le u \\ P(n) \le x^{e^{-(k+1)}}}} \frac{d_\alpha(n) f(n)}{n^{1/2-\sigma}} \Big|^2 \, \frac{du}{u^{1+\sigma}} \bigg \Vert_q^q. \end{equation}
		
		The second term in (\ref{PartialSum}), which is independent of $y$, gives a contribution in (\ref{limsup}) of
		\[ \bigg \Vert \int_1^{x^{1-e^{-(k+1)}}} \Big| \sum_{\substack{n \le z \\ P(n) \le x^{e^{-(k+1)}}}} \frac{d_\alpha(n) f(n)}{n^{1/2-\sigma}} \Big|^2 \, \frac{dz}{z^{1- 2 k / \log x + 2 \sigma}} \bigg \Vert_q^q \]
		and is thus absorbed into (\ref{BeforePl}). But now we can finally apply Lemma \ref{Plancherel} to (\ref{BeforePl}), and we obtain that it is
		\[ \ll \mathbb{E} \Bigg[ \bigg( \int_{\mathbb{R}} \frac{\left| F_k \left( \frac{1}{2} - \frac{\sigma}{2} + it \right) \right|^{2 \alpha}}{\left| \frac{\sigma}{2} + i t \right|^2} \, dt \bigg)^q \Bigg],  \]
		from which the claim follows.
	\end{proof}

	\subsection{Lower bounds}

	Next, we will obtain a lower bound for the pseudomoments in terms of moments of integrals of random Euler products. While the upper bound required considerable work in order to adapt it to the setting here, this part is a rather immediate adaptation of \cite[Proposition 3]{Harper1}.

	\begin{proposition}\label{LowerBoundProp}
		For any $1 \le \alpha \le 2$, any $0 < q < \frac{2(\alpha - 1)}{\alpha^2}$ and any sufficiently large $V$ we have
		\begin{align*}
		\Psi_{2 q, \alpha}(x) &\gg \frac{1}{(\log x)^q} \left \Vert \int_{|t| < 1} \left| F \left( \frac{1}{2} + \frac{2 V}{\log x} + i t \right) \right|^{2 \alpha} \, dt \right \Vert_q^q \\
		&- \frac{1}{e^{V q} (\log x)^q} \left \Vert \int_{\mathbb{R}} \frac{\left| F \left( \frac{1}{2} + \frac{V}{\log x} + i t \right) \right|^{2 \alpha}}{\left| \frac{V}{\log x} + i t \right|^2} \, dt \right \Vert_q^q - O(1).
		\end{align*}
	\end{proposition}

	\begin{proof}
		First, let $\epsilon$ be a Rademacher random variable, i.e. uniformly distributed on $\{ \pm 1 \}$, independent from anything else. Note that we have
		\begin{align*}
			&\mathbb{E} \bigg[ \, \bigg| \sum_{ \substack{ n \le x \\ P(n) > \sqrt{x} } } \frac{d_\alpha(n) f(n)}{\sqrt{n}} \bigg|^{2 q} \bigg] \\ = 
			\frac{1}{2^{2 q}} &\mathbb{E} \bigg[ \, \bigg| \sum_{ \substack{ n \le x \\ P(n) > \sqrt{x} } } \frac{d_\alpha(n) f(n)}{\sqrt{n}} + \sum_{ \substack{ n \le x \\ P(n) \le \sqrt{x} } } \frac{d_\alpha(n) f(n)}{\sqrt{n}} + \sum_{ \substack{ n \le x \\ P(n) > \sqrt{x} } } \frac{d_\alpha(n) f(n)}{\sqrt{n}} - \sum_{ \substack{ n \le x \\ P(n) \le \sqrt{x} } } \frac{d_\alpha(n) f(n)}{\sqrt{n}}\bigg|^{2 q} \bigg] \\
			\le \mathbb{E} \bigg[ \, &\bigg| \sum_{ \substack{ n \le x \\ P(n) > \sqrt{x} } } \frac{d_\alpha(n) f(n)}{\sqrt{n}} + \sum_{ \substack{ n \le x \\ P(n) \le \sqrt{x} } } \frac{d_\alpha(n) f(n)}{\sqrt{n}} \bigg|^{2 q} \bigg] + \mathbb{E} \bigg[ \, \bigg| \sum_{ \substack{ n \le x \\ P(n) > \sqrt{x} } } \frac{d_\alpha(n) f(n)}{\sqrt{n}} - \sum_{ \substack{ n \le x \\ P(n) \le \sqrt{x} } } \frac{d_\alpha(n) f(n)}{\sqrt{n}}\bigg|^{2 q} \bigg] \\
			= 2 \mathbb{E} &\bigg[ \, \bigg| \epsilon \sum_{ \substack{ n \le x \\ P(n) > \sqrt{x} } } \frac{d_\alpha(n) f(n)}{\sqrt{n}} + \sum_{ \substack{ n \le x \\ P(n) \le \sqrt{x} } } \frac{d_\alpha(n) f(n)}{\sqrt{n}} \bigg|^{2 q} \bigg] \\
			= 2 \mathbb{E} &\Bigg[ \mathbb{E} \bigg[ \, \bigg| \epsilon \sum_{ \substack{ n \le x \\ P(n) > \sqrt{x} } } \frac{d_\alpha(n) f(n)}{\sqrt{n}} + \sum_{ \substack{ n \le x \\ P(n) \le \sqrt{x} } } \frac{d_\alpha(n) f(n)}{\sqrt{n}} \bigg|^{2 q} \; \; \bigg| \; \; (f(p))_{p \le \sqrt{x}} \bigg] \Bigg]
			= 2 \mathbb{E} \bigg[ \, \bigg| \sum_{ n \le x } \frac{d_\alpha(n) f(n)}{\sqrt{n}} \bigg|^{2 q} \bigg],
		\end{align*}
		where the second step follows from noting that for $q \le \frac{1}{2}$ we have  $|a+b|^{2 q} \le |a|^{2 q} + |b|^{2 q} \le 2^{2 q}(|a|^{2 q} + |b|^{2 q})$, and the last step follows from the fact that the law of
		\begin{equation*} \epsilon \sum_{\substack{n \le x \\ P(n) > \sqrt{x} }} \frac{d_\alpha(n) f(n)}{\sqrt{n}} = \epsilon \sum_{ \sqrt{x} < p \le x } \frac{d_\alpha(p) f(p)}{\sqrt{p}} \sum_{m \le x/p} \frac{d_\alpha(m) f(m)}{\sqrt{m}} \end{equation*}
		conditional on $(f(p))_{p \le \sqrt{x}}$ and the law of $\sum_{ \substack{ n \le x \\ P(n) > \sqrt{x} } } \frac{d_\alpha(n) f(n)}{\sqrt{n}}$ coincide. From the Bohr correspondence, we thus see that
		\[ \Psi_{2 q, \alpha}(x) = \mathbb{E} \bigg[ \, \bigg| \sum_{ n \le x } \frac{d_\alpha(n) f(n)}{\sqrt{n}} \bigg|^{2 q} \bigg] \gg \mathbb{E} \bigg[ \, \bigg| \sum_{ \substack{ n \le x \\ P(n) > \sqrt{x} } } \frac{d_\alpha(n) f(n)}{\sqrt{n}} \bigg|^{2 q} \bigg]. \]
		Khintchine's inequality (see e.g. Lemma 3.8.1 of Gut \cite{Gut1} for the Rademacher case of this, the Steinhaus case follows by a similar argument, as is also mentioned in the proof of \cite[Proposition 3]{Harper1})
		now tells us that
		\begin{align*}
			\mathbb{E} \bigg[ \, \bigg| \sum_{ \substack{ n \le x \\ P(n) > \sqrt{x} } } \frac{d_\alpha(n) f(n)}{\sqrt{n}} \bigg|^{2 q} \bigg] &\gg \mathbb{E} \bigg[ \bigg( \sum_{ \sqrt{x} < p \le x } \frac{1}{p} \bigg| \sum_{m \le x/p} \frac{d_\alpha(m) f(m)}{\sqrt{m}} \bigg|^2 \bigg)^q \bigg],
		\end{align*}
		where we have used that $d_\alpha(p)^2 = \alpha^2 \gg 1$. Noting that $|a+b|^2 \ge \frac{1}{4}|a|^2 - |b|^2$, we can now smoothen the inner sum. Again setting $X = \exp(\sqrt{\log x})$, we deduce that this is in turn
		\begin{align}
		&= \mathbb{E} \bigg[ \bigg( \sum_{ \sqrt{x} < p \le x } \frac{X}{p^2} \int_p^{p(1+1/X)} \bigg| \sum_{m \le x/p} \frac{d_\alpha(m) f(m)}{\sqrt{m}} \bigg|^2 \, dt \bigg)^q \bigg] \nonumber \\
		&\gg \frac{1}{4^q} \mathbb{E} \bigg[ \bigg( \sum_{ \sqrt{x} < p \le x } \frac{X}{p^2} \int_p^{p(1+1/X)} \bigg| \sum_{m \le x/t} \frac{d_\alpha(m) f(m)}{\sqrt{m}} \bigg|^2 \, dt \bigg)^q \bigg] \nonumber \\
		&- \mathbb{E} \bigg[ \bigg( \sum_{ \sqrt{x} < p \le x } \frac{X}{p^2} \int_p^{p(1+1/X)} \bigg| \sum_{x/t < m \le x/p} \frac{d_\alpha(m) f(m)}{\sqrt{m}} \bigg|^2 \, dt \bigg)^q \bigg]. \label{LBSmoothing}
		\end{align}
		The only significant difference in this argument compared to the proof of \cite[Proposition 3]{Harper1} is the handling of the smoothing error here: By H\"{o}lder's inequality and orthogonality, and in the end trivially bounding the sum over primes by the sum over all integers, we have
		\begin{align*}\allowdisplaybreaks
		&\quad \; \mathbb{E} \bigg[ \bigg( \sum_{ \sqrt{x} < p \le x } \frac{X}{p^2} \int_p^{p(1+1/X)} \bigg| \sum_{x/t < m \le x/p} \frac{d_\alpha(m) f(m)}{\sqrt{m}} \bigg|^2 \, dt \bigg)^q \bigg] \\
		&\le \bigg( \sum_{ \sqrt{x} < p \le x } \frac{X}{p^2} \int_p^{p(1+1/X)} \mathbb{E} \bigg[ \, \bigg| \sum_{x/t < m \le x/p} \frac{d_\alpha(m) f(m)}{\sqrt{m}} \bigg|^2 \, \bigg] \, dt \bigg)^q \\
		&\le \bigg( \sum_{ \sqrt{x} < p \le x } \frac{1}{p} \sum_{\frac{x}{p(1+1/X)} < m \le \frac{x}{p}} \frac{d_\alpha(m)^2}{m} \bigg)^q \le \bigg( \sum_{m \le \sqrt{x}} \frac{d_\alpha(m)^2}{m} \sum_{ \frac{x}{m(1+1/X)} < p \le \frac{x}{m}} \frac{1}{p} \bigg)^q \\
		&\ll \bigg( \frac{1}{X} \sum_{m \le x} \frac{d_\alpha(m)^2}{m} \bigg)^q \ll \left(
		\frac{(\log x)^{\alpha^2}}{X} \right)^q \ll 1.
		\end{align*}
		The rest of the claim follows by a rather direct adaptation of the argument there. One bounds the first term in (\ref{LBSmoothing}) from below by
		\[ \mathbb{E} \bigg[ \bigg( \int_1^{\sqrt{x}} \bigg| \sum_{m \le z} \frac{d_\alpha(m) f(m)}{\sqrt{m}} \bigg|^2 \, \frac{dz}{z} \bigg)^q \bigg], \]
		then slightly increases the $z$-exponent from $1$ to $1 + \frac{4 V}{\log x}$ and writes $\int_1^{\sqrt{x}} = \int_1^\infty - \int_{\sqrt{x}}^{\infty}$. After some manipulations, one can then apply Parseval's identity (Lemma \ref{Plancherel}) to both terms, trivially bounding the first one from below by its contribution from say $t \in [2,4]$ and using translation-invariance in law, to deduce the claim.
		
	\end{proof}

	The exact same argument also gives the following
	
	\begin{proposition}\label{LowerBoundPropLargeAlpha}
		For any $\alpha \ge 2$, any $0 < q < \frac{1}{2}$ and any sufficiently large $V$ we have
		\begin{align*}
		\Psi_{2 q, \alpha}(x) &\gg \frac{1}{(\log x)^{q+2q(\alpha^2/4 - 1)}} \left \Vert \int_{|t| < (\log x)^{\alpha^2/4-1}} \left| F \left( \frac{1}{2} + \frac{2 V}{\log x} + i t \right) \right|^{2 \alpha} \, dt \right \Vert_q^q \\
		&- \frac{1}{e^{V q} (\log x)^q} \left \Vert \int_{\mathbb{R}} \frac{\left| F \left( \frac{1}{2} + \frac{V}{\log x} + i t \right) \right|^{2 \alpha}}{\left| \frac{V}{\log x} + i t \right|^2} \, dt \right \Vert_q^q - O(1).
		\end{align*}
	\end{proposition}

	The only difference in the proof lies in the very last step, where we bound the first integral from below by its contribution from say $t \in [2(\log x)^{\alpha^2/4-1}, 4(\log x)^{\alpha^2/4-1}] $ and use translation-invariance in law to shift it around $0$. This is suggested to give the main contribution for $\alpha > 2$ by our heuristic in the introduction.

	\section{Estimates for expectations of random Euler products}
	
	In this section, we will record some Lemmas regarding expectations of random Euler products evaluated at two fixed points (as opposed to  expectations of integrals of random Euler products over a range). By means of various applications of H\"older's inequality, we will essentially reduce to this case in the next section. The results here are immediate generalizations of known results and their proofs are mainly recorded here for convenience and completeness.
	
		\begin{lemma}\label{GMomentBounds}
		Let $\frac{3}{2} \le y < z$ with $z$ sufficiently large, let $\sigma \ge - \frac{10}{\log z}$ and define
		\[ G(t) := G_{y,z,\sigma}(t) := \prod_{y < p \le z} \left( 1 - \frac{f(p)}{p^{\frac{1}{2} + \sigma + it}} \right)^{-1}. \]
		Then for fixed $b,c \in \mathbb{R}$ we have
		\[ \mathbb{E} \left[ |G(t)|^b |G(0)|^c \right] = \exp \left( \sum_{y < p \le z} \frac{b^2+ 2 b c \cos(t \log p) + c^2}{4 p^{1+2\sigma}} + O\left( \frac{1}{\sqrt{y} \log y }\right) \right). \]
		In particular, if $|t| \ll \frac{1}{\log z}$, we have
		\[ \mathbb{E} \left[ |G(t)|^b |G(0)|^c \right] \ll \left( \frac{\log z}{\log y} \right)^{(b+c)^2/4}. \]
	\end{lemma}
	
	\begin{proof}
		The first part follows essentially from the same argument as in the proof of \cite[Euler Product Result 1]{Harper2}. Note that the parameters corresponding to $b$ and $c$ are assumed to be non-negative there, but this makes no difference in the proof. We will crucially need to apply it for negative values as well.
		
		Note also that $y$ is assumed to be sufficiently large there, depending only on $b$ and $c$. Since these parameters are assumed to be fixed here, we can absorb smaller values of $p$ into the error term as long as $z$ is sufficiently large, so that $\sigma$ can not be too negative.
		
		To deduce the second claim, note that since easily $\cos(x) = 1 + O(x)$, we have
		\[ 
		\sum_{y < p \le z} \frac{\cos ( t \log p)}{p^{1 + 2 \sigma}} = \sum_{y < p \le z} \frac{1}{p^{1+2 \sigma}} + O \left( |t| \sum_{y < p \le z} \frac{\log p}{p^{1+2 \sigma}} \right). \]
		Hence, we obtain that
		\[ \mathbb{E} \left[ |G(t)|^b |G(0)|^c \right] \ll \exp \bigg( \sum_{y < p \le z} \frac{(b+c)^2}{4 p ^{1+2 \sigma}} + O \Big( \frac{1}{\log z} \sum_{y < p \le z} \frac{\log p}{p^{1+2 \sigma}} \Big) \bigg). \]
		Now note that, since $|e^x - 1| = O(x)$ for $|x| \ll 1$, on our range of $\sigma$ we have
		\[ \sum_{y < p \le z} \frac{1}{p^{1+2 \sigma}} \le \sum_{y < p \le z} \frac{1}{p^{1- 20 / \log z}} = \sum_{y < p \le z} \frac{1}{p} + O \left( \frac{1}{\log z} \sum_{y < p \le z} \frac{\log p}{p} \right) = \log \log z - \log \log y + O(1). \]
		Since we also have
		\[ \frac{1}{\log z} \sum_{y < p \le z} \frac{\log p}{p^{1 + 2 \sigma}} \ll \frac{1}{\log z} \sum_{y < p \le z} \frac{\log p}{p} \ll 1, \]
		the second claim follows.
	\end{proof}

	Putting together Proposition \ref{REP}, Proposition \ref{REP2} and Lemma \ref{GMomentBounds}, we immediately deduce the following
	
	\begin{proposition}\label{REPLast}
		Let $\alpha \ge 1$ and $0 < q \le \frac{1}{2}$ be fixed and let $K = [\log \log \log x]$.
		Then we have
		\begin{align*}
		\Psi_{2q,\alpha}(x) \ll (\log x)^{(q \alpha)^2} + \frac{1}{(\log x)^q} \sum_{0 \le k \le K} \mathbb{E} \Bigg[ \bigg( \int_{\mathbb{R}} \frac{\left| F_{k+1} \left( \frac{1}{2} - \frac{2 (k+1)}{\log x} + it \right) \right|^{2 \alpha}}{\left| \frac{2 (k+1)}{\log x} + i t \right|^2} \, dt \bigg)^q \Bigg].
		\end{align*}
	\end{proposition}
	
	\section{Bounds for moments of integrals of random Euler products}
	
	\subsection{Upper bounds}
	
	A rather natural thing to do when arriving at Proposition \ref{REPLast} is to first subdivide the integral depending on whether $|t| \le \frac{2 (k+1)}{\log x}$ or not. Next one can dyadically decompose the range $|t| > \frac{2 (k+1)}{\log x}$, noting that on each of these intervals the denominator is roughly constant. Assuming $t > 0$ simply by symmetry in law, we are thus left with the task of bounding expressions of the type
	\[ \mathbb{E} \bigg[ \Big( \int_T^{2 T} \left| F_{k+1} \left( \frac{1}{2} - \frac{2 (k+1)}{\log x} + it \right) \right|^{2 \alpha} \, dt \Big)^q \bigg] \]
	for various sizes of $T \ge \frac{2 (k+1)}{\log x}$.
	
	Comparing with the heuristic argument in the introduction, perhaps this is a good place to point out that for general values of $k$ in Proposition \ref{REPLast}, there are not just two ranges of $T$ to consider, as was outlined in the introduction, but in fact three. The reason for this is that the Euler product $F_k$ ranges up to $x^{e^{-k}}$, and is thus expected to vary on a scale of $\frac{e^k}{\log x}$ rather than $\frac{1}{\log x}$.
	
	The first range of $T$ is therefore the range $\frac{2 (k+1)}{\log x} \ll T \ll \frac{e^k}{\log x}$, where the Euler product should be roughly constant on the whole range (or equivalently, all primes are small), and in this case it turns out not to be too difficult to bound the corresponding contribution. This will be the subject of Proposition \ref{SmallT}. The range $|t| \le \frac{2 (k+1)}{\log x}$ can be bounded in the same way, noting that by translation-invariance in law its contribution is the same as the one from $T = \frac{4 (k+1)}{\log x}$.
	
	The next range is when $\frac{e^k}{\log x} \ll T \ll 1$, and it will be dealt with in Proposition \ref{MediumT}. Here, we split the Euler product $F_k$ into the small primes $p \le \exp(1/T)$ where the Euler product factors are roughly constant, and the large primes $\exp(1/T) < p \le x^{e^{-k}}$. In order to bound the contribution of this range, the basic strategy is to define certain events which state that the Euler product over the large primes is not exceedingly large at discretised points with distance $\frac{e^k}{\log x}$. We bound the contribution to the expectation under this event by means of H\"{o}lder's inequality, although one has to be somewhat careful to apply it in an effective manner. Similar applications of H\"{o}lder's inequality, where primes are divided into small and large ones and exponents are distributed in an appropriate manner, can be found in \cite[Section 5.4]{Harper2}.
	
	We then split the complimentary event that the Euler product over large primes is large at some discretised point into several subevents according to the size of this Euler product. We then exploit the fact that these events have a very small probability by means of Chebyshev's inequality after again finding an effective way to apply H\"{o}lder's inequality.
	
	The last range is when $T \gg 1$, and is the subject of Proposition \ref{LargeT}. On this range, there are no Euler product factors whose contribution we expect to be roughly constant over the whole interval (all primes are large). We proceed in a very similar fashion to the proof of Proposition \ref{MediumT}, except that no splitting of the Euler product is necessary and we use a different bound in the definition of the event that this product is large.
	
	In the following, to shorten the notation we set
	\[ G_k(t) := F_k \left( \frac{1}{2} + \sigma + it \right). \]
	
	\begin{proposition}\label{SmallT}
		Let $\alpha \ge 1$ and $K = [ \log \log \log x]$. Suppose that $0 < q < \frac{1}{\alpha}$ and $0 \le k \le K$. Assume further that $0 < T \le \frac{e^k}{\log x}$ and that $\sigma \ge - \frac{2 (k+1)}{\log x}$. Then we have
		\[ \mathbb{E} \left[ \left( \int_T^{2 T} \left| G_k (t) \right|^{2 \alpha} \, dt \right)^q \right] \ll 
		e^{-k (q \alpha)^2} T^q (\log x)^{(q \alpha)^2}. \]
	\end{proposition}
	
	\begin{proof}
		We would like to deduce the claim from H\"{o}lder's inequality, which is particularly effective when we apply it to parts that give roughly equal contributions. Since we expect the Euler product not to change on this scale of $T$, the idea is to throw in appropriate powers of $|G_k(0)|$ (a quantity whose moments we understand very well) in order to make both parts contribute equally. Namely, we have
		\begin{align*}
		\mathbb{E} \left[ \left( \int_T^{2 T} \left| G_k(t) \right|^{2 \alpha} \, dt\right)^q \right]
		&= \mathbb{E} \left[ |G_k(0)|^{2 \alpha q (1-q)} \left( \int_T^{2 T} \left| G_k(t) \right|^{2 \alpha q} \left| \frac{G_k(t)}{G_k(0)} \right|^{2 \alpha (1-q)} \, dt\right)^q \right] \\
		&\ll \mathbb{E} \left[ |G_k(0)|^{2 \alpha q} \right]^{1-q} \mathbb{E} \left[ \int_T^{2 T} \left| G_k(t) \right|^{2 \alpha q} \left| \frac{G_k(t)}{G_k(0)} \right|^{2 \alpha (1-q)} \, dt \right]^q.
		\end{align*}
		Now Lemma \ref{GMomentBounds} gives that
		\[ \mathbb{E} \left[ |G_k(0)|^{2 \alpha q} \right] \ll (e^{-k} \log x)^{(q \alpha)^2}. \]
		Regarding the second expectation, we can simply interchange it with the integral, and the claim follows since the same Lemmas give that
		\[ \mathbb{E} \bigg[ \left| G_k(t) \right|^{2 \alpha q} \left| \frac{G_k(t)}{G_k(0)} \right|^{2 \alpha (1-q)} \bigg] \ll (e^{-k} \log x)^{(q \alpha)^2}, \]
		using that $t \ll \frac{e^k}{\log x}$.
	\end{proof}
	
	\begin{proposition}\label{MediumT}
		Let $\alpha \ge 1$ and $K = [ \log \log \log x]$. Suppose that $0 < q < \frac{1}{\alpha}$ and $0 \le k \le K$. Assume further that $\frac{e^k}{\log x} \le T \le 1$ and that $\sigma \ge - \frac{2 (k+1)}{\log x}$. Then we have
		\[ \mathbb{E} \left[ \left( \int_T^{2 T} \left| G_k(t) \right|^{2 \alpha} \, dt \right)^q \right] \ll 
		e^{-(2 \alpha - 1) k q} T^{2 \alpha q - (q \alpha)^2} (\log x)^{2 \alpha q - q}. \]
	\end{proposition}

	We remark that we expect this to be the point where we lose some powers of $\log \log x$ in the upper bound - for $1 < \alpha \le 2$ and $0 < q < \frac{2(\alpha-1)}{\alpha^2}$, where the main contribution comes from $T \asymp 1$. Namely, by means of treating different intervals of length $\frac{e^k}{\log x}$ essentially as unrelated with a union bound, we disregard the fact that the not too large primes inside the range $\exp(1/T) < p \le x^{e^{-k}}$ exhibit significant correlations over intervals much longer than $\frac{e^k}{\log x}$. To account for this properly, one would have to subdivide this range of $p$ into yet smaller intervals as is done in \cite[Section 4.1]{Harper1}, but the arising expressions seem rather difficult to control with sufficient precision for general values of $\alpha$.
	
	\begin{proof}
		Write
		\[ G^s(t) := \prod_{p \le \exp(1/T)} \left(1 - \frac{f(p)}{p^{\frac{1}{2} + \sigma + i t}} \right)^{-1} \quad \text{ and } \quad  G_k^l(t) := \prod_{\exp(1/T) < p \le x^{e^{-k}}} \left(1 - \frac{f(p)}{p^{\frac{1}{2} + \sigma + i t}} \right)^{-1}, \]
		so that $G_k(t) = G^s(t) G_k^l(t)$. Note that the first product is non-empty only because $T \ll 1$, and the second product because $T \gg \frac{e^k}{\log x}$.
		Dividing up the range of integration into intervals of length $\asymp \frac{e^k}{\log x}$, we obtain that
		\begin{equation}\label{Subdivision} \mathbb{E} \left[ \left( \int_T^{2 T} \left| G_k(t) \right|^{2 \alpha} \, dt \right)^q \right] \le \mathbb{E} \left[ \left( \sum_{e^{-k} T \log x < n \le 2 e^{-k} T \log x} \int_{|t| \le \frac{e^k}{2 \log x}} \left| G_k \left(\frac{e^k n}{\log x} + t \right) \right|^{2 \alpha} \, dt \right)^q \right], \end{equation}
		where the left and right boundary terms are to be interpreted in such a way that the range of integration on both sides coincides. 
		
		For $n \in \mathbb{Z}$, define the event
		\[ E_n = \left\{\left|G_k^l \left(\frac{e^k n}{\log x} \right) \right| \le e^{-k} T \log x \right\}, \]
		which states that the Euler product over the large primes is not exceedingly big at certain discretised points. We can then insert $\mathbbm{1}(E_n) + \mathbbm{1}(\neg \, E_n)$ in front of the integrals in (\ref{Subdivision}) and use a union bound as well as translation-invariance in law. Denoting for given $t$ by $t_{\mathrm{app}}$ a point of the form $\frac{e^k n}{\log x}$ for $n \in \mathbb{N}$ with minimal distance to $t$, we see that the right-hand side of (\ref{Subdivision}) is
		\begin{align} 
		&= \mathbb{E} \left[ \left( \sum_{e^{-k} T \log x < n \le 2 e^{-k} T \log x} \left(\mathbbm{1}(E_n) + \mathbbm{1}(\neg \, E_n) \right) \int_{|t| \le \frac{e^k}{2 \log x}} \left| G_k \left(\frac{e^k n}{\log x} + t \right) \right|^{2 \alpha} \, dt \right)^q \right] \nonumber \\
		&\le \mathbb{E} \left[ \left( \int_T^{2 T} \mathbbm{1}\left( \left|G_k^l (t_{\mathrm{app}}) \right| \le e^{-k} T \log x \right) \left| G_k(t) \right|^{2 \alpha} \, dt\right)^q \right] \nonumber \\
		&+ \sum_{e^{-k} T \log x < n \le 2 e^{-k} T \log x} \mathbb{E} \left[ \left(  \mathbbm{1}(\neg \, E_n) \int_{|t| \le \frac{e^k}{2 \log x}} \left|G_k \left(\frac{e^k n}{\log x} + t \right) \right|^{2 \alpha} \, dt \right)^q \right] \nonumber \\
		&\ll \mathbb{E} \left[ \left( \int_T^{2 T} \mathbbm{1}\left( \left|G_k^l (t_{\mathrm{app}}) \right| \le e^{-k} T \log x \right) \left| G_k(t) \right|^{2 \alpha} \, dt\right)^q \right] \nonumber \\ &+ e^{-k} T \log x \, \mathbb{E} \left[ \left(  \mathbbm{1}(\neg \, E_0) \int_{|t| \le \frac{e^k}{2 \log x}} \left| G_k(t) \right|^{2 \alpha} \, dt \right)^q \right]. \label{EventInt} \end{align}
		
		In order to bound the first term, we first use the event to pull out most of $G_k^l$ and then apply H\"{o}lder's inequality after redistributing exponents. The way we apply the inequality might not be obvious on first glance, so we try to explain the rationale behind it.
		
		Regarding the small primes, we proceed in a rather similar fashion as in the proof of Proposition \ref{SmallT}, trying to arrange them in order to give equal contributions so that H\"{o}lder is effective. Regarding the large primes, we use the nature of our event to pull out most powers of $|G_k^l(t)|$, since we expect it to be very close to $G_k^l(t_{\mathrm{app}})$. But if we would pull out everything, we would lose powers of $\log x$ by not exploiting the fact that $G_k^l$ is large only on a very short interval. Thus, we leave in the expression $|G_k^l(t_{\mathrm{app}})|^2$ in order to use this, noting that the exponent $2$ is very convenient because it is particularly good in capturing large deviations; the reader might want to compare this observation e.g. to the introduction of \cite{Harper6}, which mentions the significance of this particular exponent (for $\zeta$ on a random interval, but $F$ behaves in a similar way). 
		
		We have
		
		\begin{align*}
		&\quad \; \mathbb{E} \left[ \left( \int_T^{2 T} \mathbbm{1}\left( \left|G_k^l (t_{\mathrm{app}}) \right| \le e^{-k} T \log x \right) \left| G_k(t) \right|^{2 \alpha} \, dt\right)^q \right] \\ 
		&= \mathbb{E} \left[ \left( \int_T^{2 T} \mathbbm{1}\left( \left|G_k^l (t_{\mathrm{app}}) \right| \le e^{-k} T \log x \right) \left| G^s(t) \right|^{2 \alpha} \left| G_k^l(t_{\mathrm{app}}) \right|^{2 \alpha} \left| \frac{G_k^l(t)}{G_k^l(t_{\mathrm{app}})} \right|^{2 \alpha} \, dt\right)^q \right] \\
		&\le (e^{-k} T \log x)^{2 (\alpha - 1) q} \mathbb{E} \left[ \left( \int_T^{2 T} \left| G^s(t) \right|^{2 \alpha} \left| G_k^l(t_{\mathrm{app}}) \right|^2 \left| \frac{G_k^l(t)}{G_k^l(t_{\mathrm{app}})} \right|^{2 \alpha} \, dt\right)^q \right] \\
		&= (e^{-k} T \log x)^{2 (\alpha - 1) q} \mathbb{E} \left[ |G^s(0)|^{2 \alpha q (1-q)} \left( \int_T^{2 T} \left| G^s(t) \right|^{2 \alpha q} \left| \frac{G^s(t)}{G^s(0)} \right|^{2 \alpha (1-q)} \left| G_k^l(t_{\mathrm{app}}) \right|^2 \left| \frac{G_k^l(t)}{G_k^l(t_{\mathrm{app}})} \right|^{2 \alpha} \, dt\right)^q \right]  \\
		&\ll (e^{-k} T \log x)^{2 (\alpha - 1) q} \mathbb{E} \left[ |G^s(0)|^{2 \alpha q} \right]^{1-q} \mathbb{E} \left[\int_T^{2 T} \left| G^s(t) \right|^{2 \alpha q} \left| \frac{G^s(t)}{G^s(0)} \right|^{2 \alpha (1-q)} \left| G_k^l(t_{\mathrm{app}}) \right|^2 \left| \frac{G_k^l(t)}{G_k^l(t_{\mathrm{app}})} \right|^{2 \alpha} \, dt\right]^q.
		\end{align*}
		By Lemma \ref{GMomentBounds}, we have
		\[ \mathbb{E} \left[ |G^s(0)|^{2 \alpha q} \right] \ll T^{-(q \alpha)^2}. \]
		To deal with the second expectation, we can first interchange it with the integral, and then factor the small prime and the large prime part into two expectations using independence. Again, Lemma \ref{GMomentBounds} implies
		\[ \mathbb{E} \left[ \left| G^s(t) \right|^{2 \alpha q} \left| \frac{G^s(t)}{G^s(0)} \right|^{2 \alpha (1-q)} \right] \ll T^{-(q \alpha)^2} \]
		using that $t \ll T$, and also
		\[ \mathbb{E} \left[ \left| G_k^l(t_{\mathrm{app}}) \right|^2 \left| \frac{G_k^l(t)}{G_k^l(t_{\mathrm{app}})} \right|^{2 \alpha} \right] \ll e^{-k} T \log x \]
		using translation-invariance in law and the fact that $|t - t_{\mathrm{app}}| \ll \frac{e^k}{\log x}$. Putting these together implies the required bound for the first part in (\ref{EventInt}).
		
		To deal with the second summand in (\ref{EventInt}), we first split the event $\neg E_0$ into smaller events. Let
		\[ E_{>r} := \left\{|G_k^l(0)| \in [e^r e^{-k} T \log x, e^{r+1} e^{-k} T \log x] \right\}. \]
		A similar application of H\"{o}lder's inequality as used for the first term yields
		\begin{align*}
		&\quad \mathbb{E} \left[ \left(  \mathbbm{1}(E_{>r}) \int_{|t| \le \frac{e^k}{2 \log x}} \left| G_k(t) \right|^{2 \alpha} \, dt \right)^q \right] \\ &\ll (e^r e^{-k} T \log x)^{2 \alpha q} \mathbb{E} \left[ \mathbbm{1}(E_{>r}) \left( \int_{|t| \le \frac{e^k}{2 \log x}} |G^s(t)|^{2 \alpha} \left| \frac{G_k^l(t)}{G_k^l(0)} \right|^{2 \alpha} \, dt \right)^q \right] \\
		&= (e^r e^{-k} T \log x)^{2 \alpha q} \mathbb{E} \left[ \left( |G^s(0)|^{2 \alpha q} \mathbbm{1}(E_{>r}) \right)^{1-q} \left( \mathbbm{1}(E_{>r}) \int_{|t| \le \frac{e^k}{2 \log x}} |G^s(t)|^{2 \alpha q} \left| \frac{G^s(t)}{G^s(0)} \right|^{2 \alpha (1-q)} \left| \frac{G_k^l(t)}{G_k^l(0)} \right|^{2 \alpha} \, dt \right)^q \right] \\
		&\ll (e^r e^{-k} T \log x)^{2 \alpha q} \mathbb{E} \left[ |G^s(0)|^{2 \alpha q} \right]^{1-q} \mathbb{P}[E_{>r}]^{1-q} \mathbb{E}\left[ \mathbbm{1}(E_{>r}) \int_{|t| \le \frac{e^k}{2 \log x}} |G^s(t)|^{2 \alpha q} \left| \frac{G^s(t)}{G^s(0)} \right|^{2 \alpha (1-q)} \left| \frac{G_k^l(t)}{G_k^l(0)} \right|^{2 \alpha} \, dt \right]^q
		\\ &\ll (e^r e^{-k} T \log x)^{2 \alpha q} T^{-(q \alpha)^2} \mathbb{P}[E_{> r}]^{1-q} \left( \int_{|t| \le \frac{e^k}{2 \log x}} \mathbb{E} \left[ \mathbbm{1}(E_{>r}) \left| \frac{G_k^l(t)}{G_k^l(0)} \right|^{2 \alpha} \right] \, dt \right)^q.
		\end{align*}
		But Chebyshev's inequality implies that
		\[ \mathbb{P}[E_{>r}] \le \frac{\mathbb{E} [ |G_k^l(0)|^2 ]}{(e^r e^{-k} T \log x)^2} \ll \frac{e^{-2 r} e^k}{T \log x}  \]
		and similarly, using Lemma \ref{GMomentBounds} and noting that $\mathbbm{1}(E_{>r}) \le \frac{|G_k^l(0)|^2}{(e^r e^{-k} T \log x)^2}$,
		\[ \mathbb{E} \left[ \mathbbm{1}(E_{>r}) \left| \frac{G_k^l(t)}{G_k^l(0)} \right|^{2 \alpha} \right] \le \frac{\mathbb{E} \left[ |G_k^l(0)|^{2-2\alpha} |G_k^l(t)|^{2 \alpha} \right]}{(e^r e^{-k} T \log x)^2} \ll \frac{e^{-2r} e^k}{T \log x} \]
		for $|t| \ll \frac{e^k}{\log x}$.
		Hence, we obtain that
		\begin{align*}
		&\quad e^{-k} T \log x \, \mathbb{E} \left[ \left(  \mathbbm{1}(\neg \, E_0) \int_{|t| \le \frac{e^k}{2 \log x}} \left| G_k (t) \right|^{2 \alpha} \, dt \right)^q \right]
		\ll e^{-(2 \alpha - 1) k q} T^{2 \alpha q -(q \alpha)^2} (\log x)^{2 \alpha q - q} \sum_{r \ge 0} e^{2 (\alpha q - 1) r},
		\end{align*}
		and the latter sum converges for $q < \frac{1}{\alpha}$, thus giving the claim.
	\end{proof}
	
	\begin{proposition}\label{LargeT}
		Let $\alpha \ge 1$ and $K = [ \log \log \log x]$. Suppose that $0 < q < \frac{1}{\alpha}$ and $0 \le k \le K$. Assume further that $T \ge 1$ and that $\sigma \ge - \frac{2 (k+1)}{\log x}$. Then we have
		\[ \mathbb{E} \left[ \left( \int_T^{2 T} \left| G_k(t) \right|^{2 \alpha} \, dt \right)^q \right] \ll 
		e^{-(2 \alpha - 1) k q} T^{\alpha q} (\log x)^{2 \alpha q - q}. \]
	\end{proposition}

	Note that, since we get an extra factor $T^{-2 q}$ from the denominator of Proposition \ref{REPLast}, this result will allow us to show that for $1 \le \alpha \le 2$ the contribution of $T \ge 1$ is negligible. Employing the heuristic from the introduction, it however seems that this bound is typically does not even give the correct size on a scale of $\log x$. In particular, for $\alpha > 2$, the bound is increasing with $T$, and for these $\alpha$ and very large values of $T$ one can in fact obtain a better estimate with trivial bounds, as will be seen in section \ref{LowerBoundsFS}.
	
	\begin{proof}
		This follows by the same argument that we saw in Proposition \ref{MediumT}. The difference is that there is no need to split the Euler product, so $G_k^l = G_k$ (and $G^s = 1$). Also, the right events to define in this case are
		\[ E_n = \left\{ \left| G_k \left( \frac{e^k n}{\log x} \right) \right| \le e^{-k} T^{1/2} \log x \right\}, \]
		making the corresponding bounds under the event that each of these holds roughly the same as the complementary event. The details are left to the reader.
	\end{proof}

	\begin{proof}[Proof of Theorem 1.]
		By Proposition \ref{REPLast} and translation-invariance in law, we have
		\[ \Psi_{2q,\alpha}(x) \ll (\log x)^{(q \alpha)^2} + \sum_{0 \le k \le K} \sum_{\substack{T \ge \frac{2 (k+1)}{\log x} \\ T \text{ dyadic}} } \frac{1}{(\log x)^q T^{2 q}} \mathbb{E} \left[ \left( \int_T^{2 T} | G_k (t) |^{2 \alpha} \, dt \right)^q \right]. \]
		Plugging in Propositions \ref{SmallT}, \ref{MediumT} and \ref{LargeT} on the respective ranges then gives the claim. Note that the main contribution comes from $T \asymp \frac{k+1}{\log x}$ precisely when $\frac{2(\alpha - 1)}{\alpha ^2} < q \le \frac{1}{2}$, and from $T \asymp 1$ precisely when $0 < q < \frac{2(\alpha-1)}{\alpha^2}$. At $q = \frac{2(\alpha-1)}{\alpha^2}$, the whole range of $\frac{k+1}{\log x} \ll T \ll 1$ gives the same contribution, thus giving an extra factor $\log \log x$.
	\end{proof}

	\subsection{Lower bounds}\label{LowerBoundsFS}
	
	We will now complete the proof of Theorem \ref{LowerBoundThm} using Propositions \ref{LowerBoundProp} and \ref{LowerBoundPropLargeAlpha}. The arguments follow ideas from \cite[Section 6]{Harper1} and \cite[Section 6]{Harper3}.

	\begin{proof}[Proof of Theorem 2 for $1 \le \alpha \le 2$.]
		We will closely follow \cite[Section 6]{Harper1}. Applying the lower bound in Proposition \ref{LowerBoundProp}, we begin by upper-bounding the subtracted term in a rather trivial way. Namely, assuming that $V \ge 1$, H\"{o}lder's inequality and Lemma \ref{GMomentBounds} imply that
		\begin{align*}
		&\quad \; \frac{1}{e^{V q} (\log x)^q} \mathbb{E} \Bigg[ \bigg (\int_{\mathbb{R}} \frac{\left| F \left( \frac{1}{2} + \frac{V}{\log x} + i t \right) \right|^{2 \alpha}}{\left| \frac{V}{\log x} + i t \right|^2} \, dt \bigg)^q \Bigg] \\
		&\le e^{-V q} \sum_{ \substack{ T \ge \frac{V}{\log x} \\ T \text{ dyadic} } } \frac{1}{(\log x)^q T^{2 q}} \mathbb{E} \Bigg[ \bigg (\int_T^{2 T} \left| F \left( \frac{1}{2} + \frac{V}{\log x} + i t \right) \right|^{2 \alpha} \, dt \bigg)^q \Bigg] \\
		&\le e^{-V q} \sum_{ \substack{ T \ge \frac{V}{\log x} \\ T \text{ dyadic} } } \frac{1}{(\log x)^q T^{2 q}} \Bigg( \int_T^{2 T} \mathbb{E} \left[\, \left| F \left( \frac{1}{2} + \frac{V}{\log x} + i t \right) \right|^{2 \alpha} \, \right] \, dt \Bigg)^q \\
		&\ll e^{-V q} (\log x)^{q(\alpha^2 - 1)} \sum_{ \substack{ T \ge \frac{V}{\log x} \\ T \text{ dyadic} } } T^{-q} \ll e^{-V q} (\log x)^{q \alpha^2}.
		\end{align*}
		Taking say $V = (\alpha^2 - 2 \alpha + 3) \log \log x$, this is $\ll (\log x)^{2(\alpha - 1) q - q}$. Note that for $1 \le \alpha \le 2$ we could get much better bounds from Propositions \ref{SmallT}, \ref{MediumT} and \ref{LargeT}, which would allow us in the end to take $V \asymp \log \log \log x$ and thus give somewhat better lower bounds. Since we expect the arising lower bound to still be far away from the truth in terms of the exponent of $\log \log x$, we have refrained from doing so.
		
		It thus remains to obtain a lower bound for
		\[ \frac{1}{(\log x)^q} \left \Vert \int_{|t| < 1} \left| F \left( \frac{1}{2} + \frac{2 V}{\log x} + i t \right) \right|^{2 \alpha} \, dt \right \Vert_q^q \]
		with our choice of $V$ above. To this end, note that since $\exp$ is convex, Jensen's inequality implies that
		\begin{align*}
		&\quad \; \int_{|t| < 1} \left| F \left( \frac{1}{2} + \frac{2 V}{\log x} + i t \right) \right|^{2 \alpha} \, dt \\ &\ge \sum_{|n| \le \log x - 1} \int_{|t| < \frac{1}{2 \log x}} \left| F \left( \frac{1}{2} + \frac{2 V}{\log x} + i \frac{n}{\log x} + i t \right) \right|^{2 \alpha} \, dt \\
		&= \frac{1}{\log x} \sum_{|n| \le \log x - 1} \Bigg( \log x \int_{|t| < \frac{1}{2 \log x}} \exp \left( 2 \alpha \log \left| F \left( \frac{1}{2} + \frac{2 V}{\log x} + i \frac{n}{\log x} + i t \right) \right| \right) \, dt \Bigg)  \\
		&\ge \frac{1}{\log x} \sum_{|n| \le \log x - 1} \exp \bigg( 2 \alpha \log x \int_{|t| < \frac{1}{2 \log x}} \log \left| F \left( \frac{1}{2} + \frac{2 V}{\log x} + i \frac{n}{\log x} + i t \right) \right| \, dt \bigg).
		\end{align*}
		Next, we can compute
		\begin{align*}
			\log \left| F \left( \frac{1}{2} + \frac{2 V}{\log x} + i \frac{n}{\log x} + i t \right) \right| &= - \Re \left( \sum_{p \le x} \log \left( 1 - \frac{f(p)}{p^{\frac{1}{2} + \frac{2 V}{\log x} + i \frac{n}{\log x} + i t}} \right) \right) \\
			&= \sum_{p \le x} \Bigg( \frac{\Re( f(p) p^{- i \frac{n}{\log x} - i t })}{p^{\frac{1}{2} + \frac{2 V}{\log x}}} + \frac{\Re( f(p)^2 p^{- i \frac{ 2 n}{\log x} - 2 i t })}{2 p^{1 + \frac{4 V}{\log x}}} \Bigg) + O(1).
		\end{align*}
		Further, we remark that
		\[ \log x \int_{|t| < \frac{1}{2 \log x}} p^{-2 i t} \, dt = \log x \int_{|t| < \frac{1}{2 \log x}} \left( 1 + O(|t| \log p) \right) \, dt = 1 + O \left( \frac{\log p}{\log x} \right) \]
		and
		\[ \sum_{p \le x} \frac{\log p}{p^{1+\frac{4 V}{\log x}} \log x} = O(1), \]
		so that all together, bounding the sum over $n$ from below by the maximum, we have
		\begin{align*}
		&\quad \; \Psi_{2 q, \alpha}(x) \\ &\gg \frac{1}{(\log x)^{2 q}} \mathbb{E} \Bigg[ \Bigg( \max_{|n| \le \log x - 1} \exp \bigg( \log x  \sum_{p \le x} \int_{|t| < \frac{1}{2 \log x}}  \frac{\Re( f(p) p^{- i \frac{n}{\log x} - i t })}{p^{\frac{1}{2} + \frac{2 V}{\log x}}}  \, dt + \sum_{p \le x} \frac{\Re( f(p)^2 p^{- i \frac{ 2 n}{\log x}})}{2 p^{1 + \frac{4 V}{\log x}}} \bigg) \Bigg)^{2 \alpha q} \Bigg] \\
		&- O \left( (\log x)^{2(\alpha - 1) q - q} \right).
		\end{align*}
		But for any non-negative random variable $X$ and any fixed $y> 0$ we have $\mathbb{E}[X^q] \ge y^q \mathbb{P}[X \ge y].$
		Hence it suffices to show that for some constant $C > 0$ we have
		\begin{align*} \mathbb{P} \Bigg[ &\max_{|n| \le \log x - 1} \bigg( \log x  \sum_{p \le x} \int_{|t| < \frac{1}{2 \log x}}  \frac{\Re( f(p) p^{- i \frac{n}{\log x} - i t })}{p^{\frac{1}{2} + \frac{2 V}{\log x}}}  \, dt + \sum_{p \le x} \frac{\Re( f(p)^2 p^{- i \frac{ 2 n}{\log x}})}{2 p^{1 + \frac{4 V}{\log x}}} \bigg) \Bigg. \\ \Bigg. &\ge \log \log x - \frac{1}{2} \log \log \log  x - C \Bigg] \gg \frac{1}{(\log \log x)^6}.
		\end{align*}
		Now
		\[ \left| \sum_{p \le (\log x)^{10}} \frac{\Re( f(p)^2 p^{- i \frac{ 2 n}{\log x}})}{2 p^{1 + \frac{4 V}{\log x}}} \right| \le \sum_{p \le (\log x)^{10}} \frac{1}{2 p} \le \frac{1}{2} \log \log \log x + O(1) \]
		and
		\[ \mathbb{E} \bigg[ \,  \bigg| \sum_{ (\log x)^{10} < p \le x }\frac{\Re( f(p)^2 p^{- i \frac{ 2 n}{\log x}})}{2 p^{1 + \frac{4 V}{\log x}}}  \bigg|^2 \, \bigg] \le \sum_{ (\log x)^{10} < p \le x } \frac{1}{p^2} \ll (\log x)^{-10}, \]
		so that by union bound and Chebyshev's inequality we have
		\[ \mathbb{P} \bigg[ \max_{|n| \le \log x - 1} \bigg| \sum_{ (\log x)^{10} < p \le x} \frac{\Re( f(p)^2 p^{- i \frac{ 2 n}{\log x}})}{2 p^{1 + \frac{4 V}{\log x}}}\bigg| \ge 1 \bigg] \ll (\log x)^{-9}.  \]
		As a consequence, it suffices to show in turn that
		\[ \mathbb{P} \Bigg[ \max_{|n| \le \log x - 1} \bigg( \log x  \sum_{p \le x} \int_{|t| < \frac{1}{2 \log x}}  \frac{\Re( f(p) p^{- i \frac{n}{\log x} - i t })}{p^{\frac{1}{2} + \frac{2 V}{\log x}}}  \, dt \bigg) \ge \log \log x \Bigg] \gg \frac{1}{(\log \log x)^6}. \]
		Lastly, from independence and the fact that
		\[ \mathbb{P} \bigg[ \log x \sum_{p \le (\log x)^{10}} \int_{|t| < \frac{1}{2 \log x}}  \frac{\Re( f(p) p^{- i \frac{n}{\log x} - i t })}{p^{\frac{1}{2} + \frac{2 V}{\log x}}}  \, dt \ge 0 \bigg] = \frac{1}{2}, \]
		we can furthermore omit the primes $p \le (\log x)^{10}$. We thus arrive at \cite[(6.3)]{Harper1} (with $\beta = 2$ say), the only difference being that our denominator is $p^{\frac{1}{2} + \frac{2(\alpha^2 - 2 \alpha + 3) \log \log x}{\log x} }$ instead of $p^{\frac{1}{2} + \frac{\log \log x}{\log x} }$.
		The claim now follows from the proof there. The reader is invited to verify that the slightly bigger shift makes no difference for the final bound, and that the $(\log \log x)^{O(1)}$ in the denominator there is indeed $(\log \log x)^6$.
		
		The basic idea for proving this claim is to use a multivariate central limit theorem in order to replace for each $|n| \le \log x - 1$ the sum by a Gaussian random variable with the same mean and variance so that the covariance structure also remains. One then uses lower bound results on the maximum of Gaussian processes \cite[Theorem 1]{Harper3}. This will in fact be spelled out in more detail in the next proof, since for $\alpha > 2$ we need to modify some of the parameters to make the argument work. 
	\end{proof}

	\begin{proof}[Proof of Theorem 2 for $\alpha > 2$.]
		We closely imitate the previous proof, although in particular in the final part we have to work a bit more. 
		
		Firstly, Proposition \ref{LowerBoundPropLargeAlpha} and the same argument as above, this time taking say $V = \left(\frac{\alpha^2}{2}+ 1 \right) \log \log x$, implies that
		\begin{align*}
		\Psi_{2q, \alpha}(x) \gg (\log x)^{q-q \alpha^2/2} \left \Vert \int_{|t| < (\log x)^{\alpha^2/4-1}} \left| F \left( \frac{1}{2} + \frac{2 V}{\log x} + i t \right) \right|^{2 \alpha} \, dt \right \Vert_q^q - O \left( (\log x)^{q \alpha^2/2 - q} \right).
		\end{align*}
		From the exact same argument as in the last proof it also follows, again using Jensen's inequality, that it suffices this time to prove that for $\alpha > 2$ we have
		\begin{align*} \mathbb{P} \Bigg[ &\max_{0 \le m \le (\log x)^{\alpha^2/4-1}}  \max_{|n| \le \log x - 1} \bigg( \log x  \sum_{p \le x} \int_{|t| < \frac{1}{2 \log x}}  \frac{\Re( f(p) p^{- i m - i \frac{n}{\log x} - i t })}{p^{\frac{1}{2} + \frac{2 V}{\log x}}}  \, dt + \sum_{p \le x} \frac{\Re( f(p)^2 p^{- 2 i m - i \frac{ 2 n}{\log x}})}{2 p^{1 + \frac{4 V}{\log x}}} \bigg) \Bigg. \\ \Bigg. &\ge \frac{\alpha}{2} \log \log x - \frac{1}{2} \log \log \log  x - C \Bigg] \gg \frac{1}{(\log \log x)^{2 \alpha^2 / 3 + 5}}.
		\end{align*}
		We can also copy the previous argument to in turn reduce it to showing that
		\begin{align} &\mathbb{P} \Bigg[ \max_{0 \le m \le (\log x)^{\alpha^2/4-1}}  \max_{|n| \le \log x - 1} \bigg( \log x  \sum_{y < p \le x} \int_{|t| < \frac{1}{2 \log x}}  \frac{\Re( f(p) p^{- i m - i \frac{n}{\log x} - i t })}{p^{\frac{1}{2} + \frac{2 V}{\log x}}}  \, dt \bigg) \ge \frac{\alpha}{2} \log \log x \Bigg] \nonumber \\ &\gg \frac{1}{(\log \log x)^{2 \alpha^2 / 3 + 5}} \label{Indep}
		\end{align}
		for some parameter $2 \le y \le x$. We would like to choose $y = (\log x)^C$ for some appropriate constant $C$, but due to limitations coming from the error term in the prime number theorem, this will not be possible. Instead, we set
		\[ y = \exp \left( c (\log \log x)^{\frac{5}{3} + \varepsilon} \right) \]
		for some constants $c, \, \varepsilon > 0$ (both depending on $\alpha$), where $\varepsilon$ will be chosen later. Note that this implies that \[\log y \asymp (\log \log x)^{\frac{5}{3}+\varepsilon} \quad \text{ and } \quad \log \log y = \left(\frac{5}{3}+ \varepsilon \right) \log \log \log x + O(1).\] However, even if we could choose $y$ as small as a power of $\log x$, say under the Riemann Hypothesis this will turn out to be possible, the resulting bound will likely still be far from optimal in terms of the exponent of $\log \log x$, so we refrain from giving a conditional bound as well.
		
		We will prove (\ref{Indep}) by first proving that for any $0 \le m \le (\log x)^{\alpha^2/4-1}$ we have
		\begin{align} &\mathbb{P} \Bigg[ \max_{|n| \le \log x - 1} \bigg( \log x  \sum_{y < p \le x} \int_{|t| < \frac{1}{2 \log x}}  \frac{\Re( f(p) p^{- i m - i \frac{n}{\log x} - i t })}{p^{\frac{1}{2} + \frac{2 V}{\log x}}}  \, dt \bigg) \ge \frac{\alpha}{2} \log \log x \Bigg] \nonumber \\ &\gg \frac{1}{(\log x)^{\alpha^2/4-1}(\log \log x)^{2 \alpha^2 / 3 + 5}} \label{IntClaim}
		\end{align}
		and then, taking only every second interval to make sure that the correlations are small, asserting that for different $m$ the events are almost independent. 
		
		The proof of (\ref{IntClaim}) is in fact a rather direct adaptation of the proof of \cite[(6.3)]{Harper1}, one only needs to change some of the parameters. Namely, we can replace (for fixed $m$) the family of random variables
		\[ \bigg( \log x  \sum_{y < p \le x} \int_{|t| < \frac{1}{2 \log x}}  \frac{\Re( f(p) p^{- i m - i \frac{n}{\log x} - i t })}{p^{\frac{1}{2} + \frac{2 V}{\log x}}}  \, dt \bigg)_{|n| \le \log x - 1}  \]  
		by Gaussian random variables $(X(n))_{|n| \le \log x - 1}$ with the same means and covariances by the multivariate central limit theorem. Moreover, one computes (compare the proof of \cite[(6.3)]{Harper1}) that $\mathbb{E}X(n) = 0$ and
		\begin{equation} \mathrm{Cov}(X(n_1),X(n_2)) = \sum_{y < p \le x^{1/\log \log x}} \frac{\cos \left( \frac{|n_1-n_2| \log p}{\log x} \right)}{2 p} + O(1). \label{Cov} \end{equation}
		Hence, for $|n_1 - n_2| \le \log \log x$ we have
		\begin{align*} \mathrm{Cov}(X(n_1),X(n_2)) &= \sum_{y < p \le x^{1/\log \log x}} \frac{1 + O \left( \frac{|n_1 - n_2|^2 \log^2 p}{\log^2 x} \right)}{2 p} \\ &= \frac{1}{2} \left(\log \log x - \left( \frac{8}{3} + \varepsilon \right) \log \log \log x \right) + O(1). \end{align*}
		Next, we will use the following form of the prime number theorem (compare e.g. \cite[p. 194]{MontVaughan} and the references therein)
		\[ \pi(x) = \mathrm{li}(x) + O_\delta \left( x \exp \left( - (\log x)^{\frac{3}{5}-\delta} \right) \right) \]
		for $\delta > 0$. Following the argument in \cite[Section 6.1]{Harper3}, but using this stronger version of the prime number theorem, gives that
		\begin{equation}\label{PNT} \sum_{y < p \le z } \frac{\cos ( t \log p)}{p} = \int_{t \log y}^{t \log z} \frac{\cos u}{u} \, du + O_\delta \left((1+|t|) \exp\left(- (\log y)^{\frac{3}{5}-\delta}\right) \right). \end{equation}
		Plugging this into (\ref{Cov}) with our choice of $y$ and with $z = x^{1/\log \log x}$ implies that for $\log \log x \le |n_1 - n_2| \le \frac{\log x}{\log y}$ we have
		\[ \mathrm{Cov}(X(n_1),X(n_2)) = \frac{1}{2} \left( \log \log x - \log |n_1 - n_2| - \left( \frac{5}{3} + \varepsilon \right) \log \log \log x \right) + O(1).  \]
		Putting our estimates so far together, again in a similar way as in the proof of \cite[(6.3)]{Harper1}, implies that for any large parameter $E \in \mathbb{N}$ we have
		\[ \mathbb{P} \left[ \max_{|n| \le \log x - 1} X(n) \ge \frac{\alpha}{2} \log \log x \right] \ge \mathbb{P} \left[ \max_{|j| \ll \frac{\log x}{E (\log \log x)^{8/3+\varepsilon}}} \frac{X(j E [\log \log x])}{\sqrt{\mathbb{E}[ X(j E [\log \log x])^2 ]}} \ge u \right],  \]
		where this time
		\[ u = \frac{\frac{\alpha}{2} \log \log x}{\sqrt{\frac{1}{2} \left(\log \log x - \left( \frac{8}{3} + \varepsilon \right) \log \log \log x \right) + O(1)}}. \]
		Setting \[Z(j) = \frac{X(j E [\log \log x])}{\sqrt{\mathbb{E}[ X(j E [\log \log x])^2 ]}},\] the reader is invited to verify that \cite[Theorem 1]{Harper3} is also applicable here, and that setting $E = [\sqrt{\log \log x}]$ is also a possible choice, and that this implies that
		\begin{align*}
		&\mathbb{P} \left[ \max_{|j| \ll \frac{\log x}{E (\log \log x)^{8/3+\varepsilon}}} Z(j) \ge u \right] \gg \frac{(\log x) e^{-u^2/2}}{E (\log \log x)^{8/3+\varepsilon} u^3 } \gg \frac{1}{(\log x)^{\alpha^2/4-1}(\log \log x)^{2 \alpha^2 / 3 + 5}}
		\end{align*}
		by choosing $\varepsilon = \frac{1}{3(\alpha^2/4 + 1)}$, therefore giving (\ref{IntClaim}).
		
		To deduce (\ref{Indep}) from this, we closely follow the argument in \cite[Section 6.3]{Harper3}. To this end, for $0 \le n \le N = [\log x / E \log y]$ and $0 \le m \le M = [(\log x)^{\alpha^2/4-1} / 2]$, set \[\tilde Z(m,n) = \frac{X(n E + 2 m [\log x])}{\sqrt{\mathbb{E}[ X(n E + 2 m [\log x])^2 ]}} .\] We now apply \cite[Comparison Inequality 1]{Harper3}, and remark that this requires $y$ to be at least a sufficiently large power of $\log x$ (which it is) in order to ensure that the correlations are indeed small for different values of $m$. Using (\ref{PNT}) to estimate these correlations as well as using translation-invariance in law, the comparison inequality implies that
		\begin{align*}
		&\quad \; \left| \mathbb{P} \left[ \max_{0 \le m \le M} \max_{ 0 \le n \le N } \tilde Z(m,n) \le u \right] - \prod_{0 \le m \le M} \mathbb{P} \left[ \max_{0 \le n \le N } \tilde Z(m,n) \le u\right] \right| \\
		&\ll \left(\frac{\log y}{\log x}\right)^{\frac{\alpha^2}{2}} \sum_{ 0 \le m_1 < m_2 \le M } \sum_{0 \le n_1, n_2 \le N}  \left| \mathbb{E} \left[ \tilde Z(m_1,n_1) \tilde Z(m_2,n_2) \right] \right| \\
		&\ll_\delta \frac{(\log y)^{\frac{\alpha^2}{2}} N^2 }{(\log x)^{\frac{\alpha^2}{2}} (\log \log x) } \sum_{0 \le m_1 < m_2 \le M} \left( \frac{1}{|m_1-m_2| \log y } + (m_1+m_2) \exp \left(-(\log y)^{\frac{3}{5} - \delta} \right) \right) \\
		&\ll \left( \frac{\log y}{\log x} \right)^{\frac{\alpha^2}{2} - 2} \left( \frac{M \log M}{\log y} + M^3 \exp \left(-(\log y)^{\frac{3}{5}-\delta} \right) \right)
		\ll (\log x)^{1 - \alpha^2/4 + \varepsilon'} + (\log x)^{-\varepsilon'}
		\end{align*}
		for some $\varepsilon' > 0$ (depending on $\alpha$) by choosing $\delta > 0$ sufficiently small and $c$ in the definition of $y$ sufficiently large (both depending on $\alpha$). The claim follows.
	\end{proof}

	\bibliography{PM}
	\bibliographystyle{alpha}
	
\end{document}